\documentclass[11pt,reqno]{amsart}
\usepackage{color}
\usepackage[active]{srcltx}
\usepackage{a4wide}
\usepackage{amssymb, amsmath, mathtools}
\usepackage{graphicx}
\usepackage{float}
\usepackage[active]{srcltx}
\usepackage[ansinew]{inputenc}
\usepackage{hyperref}

\theoremstyle{plain}% default
\newtheorem{theorem}{Theorem}[section]
\newtheorem{maintheorem}{Theorem}
\newtheorem{lemma}[theorem]{Lemma}
\newtheorem{proposition}[theorem]{Proposition}

\theoremstyle{remark}

\newtheorem{definition}{Definition}

\newtheorem{remark}[theorem]{Remark}%[section]

\numberwithin{equation}{section}

\newcommand{\NN}{{\mathbb{N}}}
\newcommand{\ZZ}{{\mathbb{Z}}}
\newcommand{\RR}{{\mathbb{R}}}
\newcommand{\EU}{{\mathbb{S}}}

\newcommand{\In}{{\text{In}}}
\newcommand{\Out}{{\text{Out}}}
\newcommand{\Fix}{{\text{Fix}}}

\newcommand{\dpt}{\displaystyle}

\begin{document}

\title[Strange attractors near an attracting periodically-perturbed network]{Abundance of strange attractors \\ near an attracting periodically-perturbed network}
\author[Alexandre A. P. Rodrigues]{Alexandre A. P. Rodrigues \\ Centro de Matem\'atica da Univ. do Porto \\ Rua do Campo Alegre, 687,  4169-007 Porto,  Portugal }
\address{Alexandre Rodrigues \\ Centro de Matem\'atica da Univ. do Porto \\ Rua do Campo Alegre, 687 \\ 4169-007 Porto \\ Portugal}
\email{alexandre.rodrigues@fc.up.pt}

\date{\today}

\thanks{ AR was partially supported by CMUP (UID/MAT/00144/2019), which is funded by FCT with national (MCTES) and European structural funds through the programs FEDER, under the partnership agreement PT2020. AR also acknowledges financial support from Program INVESTIGADOR FCT (IF/00107/2015).}

\subjclass[2010]{ 34C28; 34C37; 37D05; 37D45; 37G35 \\
\emph{Keywords:}  May-Leonard network, periodic forcing, bifurcations, rank-one strange attractors, abundance.}

\begin{abstract} 
We study the  dynamics of the periodically-forced May-Leonard system. We extend  previous results on the field and we identify different dynamical regimes depending  on the strength of attraction $\delta$ of the network and the frequency $\omega$ of the periodic forcing.
We focus our attention in the case $\delta\gg1$ and $\omega \approx 0$, where we show that,  for a positive Lebesgue measure set of parameters (amplitude of the periodic forcing), the dynamics are dominated by strange attractors with fully stochastic properties, supporting Sinai-Ruelle-Bowen (SRB) measures. 
The proof is performed by using the Wang and Young \emph{Theory of rank-one strange attractors}. This work ends the discussion about the existence of observable and sustainable chaos in this scenario.  We also identify some bifurcations occurring in the transition from an attracting two-torus to rank-one strange attractors, whose existence has been suggested by numerical simulations.

\end{abstract}

\maketitle \setcounter{tocdepth}{1}
%\tableofcontents

\section{Introduction}\label{intro}
Many aspects contribute to the richness and complexity of a dynamical system. One of them is the existence of \emph{strange attractors} (observable chaos). Before going further, we introduce the following notion:
\begin{definition}
A  (H\'enon-type) \emph{strange attractor} of a two-dimensional dissipative diffeomorphism, defined on a Riemannian manifold, is a compact invariant set $\Lambda$ with the following properties:
\begin{enumerate}
\item $\Lambda$  equals the closure of the unstable manifold of a hyperbolic periodic point;
%\medbreak
\item the basin of attraction of $\Lambda$  contains an open set;
%\medbreak
\item there is a dense orbit in $\Lambda$ with a positive Lyapounov exponent;
\item $\Lambda$ is not hyperbolic.
%\medbreak
\end{enumerate}
A vector field possesses a \emph{strange attractor }if the first return map to a cross section does.  
\end{definition}
The rigorous proof of the strange character of an invariant set is a great challenge and the proof of the \emph{persistence} (with respect to the Lebesgue measure) of such attractors is a very involved task. In the present paper, rather than exhibit the existence of strange attractors, we explore a mechanism to obtain them near a periodically-perturbed vector field whose unperturbed dynamics  exhibit an attracting heteroclinic network. 
The persistence of chaotic dynamics is physically relevant because it means that the phenomenon is numerically \emph{observable} with positive probability. 

The notion of SRB (Sinai-Ruelle-Bowen) measure has evolved as the theory of nonuniform hyperbolicity has developed. The following concept, adapted to our purposes, is important throughout this article:
\begin{definition}
Let $F$ be a two-dimensional $C^2$ dissipative diffeomorphism defined on a Riemannian compact manifold. An invariant Borel probability measure $\nu$ for $F$ is called an \emph{SRB measure} if $F$ has a positive Lyapunov exponent $\nu$-almost everywhere and the conditional measures of $\nu$ on unstable manifolds are equivalent to the Riemannian volume on these leaves. 
\end{definition}

Strange attractors supporting SRB measures are of fundamental importance in dynamical systems; they have
been observed and recognized in many scientific disciplines.  Among the examples that have been studied are the Lorenz and Hénon attractors, both of which are closely related to suitable one-dimensional maps (cf. \cite{AP2000, MV93, YB93}). 

For families of autonomous differential equations in $\RR^3$, a typical context for several realistic models, the persistence of strange attractors can be proved near heteroclinic networks whose first return map to a cross section has a homoclinic tangency to a dissipative saddle \cite{{Homb2002}, LR2015, MV93, OS}. To date there has been very little systematic investigation
of the effects of perturbations that are time-periodic, although they are natural for the modelling of \emph{seasonal effects} on physical and biological models (see \cite{Barrientos_SIR, LR2020} and references therein).

Based on numerics presented in \cite{DT3, TD2}, the main goal of this article is to provide an analytic criterion for the existence of persistent  strange attractors near the forced May-Leonard system, using the \emph{Theory of rank-one maps}\footnote{This theory generalizes the methodology used in \cite{BC91, MV93} on the existence of H\'enon attractors.}.
This theory,  developed by Q. Wang and L.-S. Young \cite{WY2001, WY, WY2003, WY2008}, has experienced unprecedented
growth in the last two decades and provides checkable conditions that imply the existence of nonuniformly hyperbolic dynamics and SRB measures in parametrized families $F_{\gamma}$ of dissipative embeddings in $\RR^n$ for any $n \geq 2$.  The theory asserts that, under certain checkable conditions, there exists a set $\Delta\subset \RR$ of values with positive Lebesgue measure such that if $\gamma \in \Delta$, then $F_\gamma$ has a strange attractor supporting an ergodic SRB measure. The term \emph{rank-one} refers to the local character of the embeddings: some instability in one direction and strong contraction in the other direction.

We bring some of the
techniques considered in \cite{WY2001, WY, WY2003}  
 to study bifurcations near the forced May-Leonard system (whose unperturbed flow contains an attracting and clean network).  The periodic forcing is biologically significant for predator-prey models.  We revive the proof of \cite{WY2001} making an interpretation of the hypotheses in terms of the initial vector field.

\subsection{The object of study}
For $0\leq \gamma \ll 1$ and $\omega\in \RR^+$, the focus of this article is the model given by the following set of the ordinary differential equations with a periodic forcing (also called the \emph{forced May-Leonard system}):
\begin{equation}
\label{general}
\left\{ 
\begin{array}{l}
\dot x = x((1-r)-cy+ez) +\gamma(1-x)\sin^2(2\omega t)\\ \\
\dot y =  y((1-r)-cz+ex)\\\\
\dot z = z((1-r)-cx+ey)\\ \\
x(0)>0, \, \, y(0)>0, \, \, z(0)>0
\end{array}
\right.
\end{equation}
where $r= x+y+z$, $\omega>0$ and\\
 \begin{description}
\item[(C1a)] $ 0<e<c<1.$ \\
\item[(C1b)] There exist $d_1, d_2\in \RR^+$ such that for all $m,n \in \ZZ$, the following inequality holds:
$$
|m\, c-n\, e|> d_1 (\, |m|+|n|\, )^{-d_2}.
$$

\label{condition_attraction}
 \end{description}

Conditions \textbf{(C1a)} and  \textbf{(C1b)} define an open  subset of $\RR^2$, for the usual topology. 
Concerning the equation \eqref{general}, the amplitude of the perturbing term is governed by $\gamma$, which is supposed to be small ($0\leq \gamma \ll 1$). 
Our choice of perturbing term $\gamma(1-x)\sin^2(2\omega t)$ in the \emph{radial direction} of the equilibria $(\pm 1, 0, 0)$ is made for two reasons: it simplifies the computations and allows comparison with previous work by other authors \cite{AHL2001, TD2, TD}.  %See Remark \ref{rem:generalization1}. 
From now on, let us denote the one-parameter family of vector fields associated to \eqref{general} by $f_\gamma$.

\subsection{The unperturbed system ($\gamma=0$)}
\label{gamma=0}
The vector field $f_0$ is exactly the \emph{toy-model} example proposed
by May and Leonard \cite{ML75} as an example of competitive
Lotka-Volterra model for the dynamics of three
populations (\emph{winnerless competition}); $x(t)$, $y(t)$ and $z(t)$ are the non-negative proportions of the total population that consists of each species. %The robustness of the cycle is forced by the structure of the Lotka-Volterra model.  
\medbreak

\begin{figure}[ht]
\begin{center}
\includegraphics[height=5.5cm]{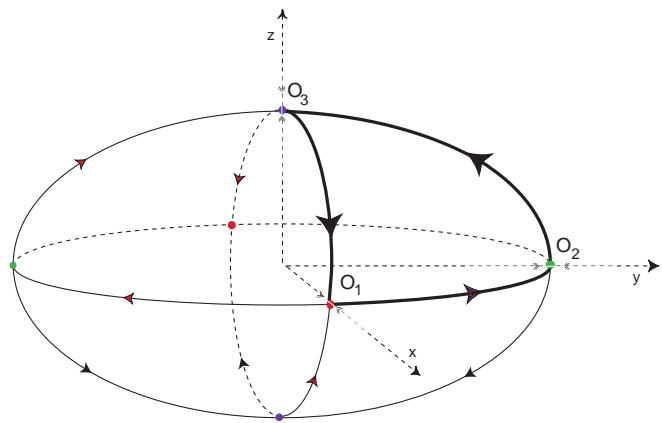}
\end{center}
\caption{\small May and Leonard network. Schematic flow of \eqref{general} for $\gamma=0$. In bold, it is stressed the cycle $\Gamma_1$, the restriction of the network to the first octant. }
\label{GH1}
\end{figure}

Guckenheimer and Holmes   \cite{GH88} studied a $\Theta$-equivariant vector field, where   $\Theta \subset \mathbb{O}(3)$ is the finite Lie group generated by:
$$\ell(x,y,z)=(y,z, x)$$ and 
$$
k_1(x, y, z)=(-x, y, z), \qquad k_2(x, y, z)=(x, -y, z),  \qquad k_3(x, y, z)=(x, y, -z),
$$
whose action on $\RR^3$ is isomorphic to  $\ZZ_3 \dot\ltimes   \ZZ^2_3$. The coordinate planes and axes are flow-invariant; they correspond to $\text{Fix}\,\, \ZZ_2(k_1)$, $\text{Fix}\,\, \ZZ_2(k_2)$ and $\text{Fix}\, \,  \ZZ_2(k_3)$ and their (mutual) intersections. The Birkhoff normal form of a $\Theta$-equivariant vector field at the origin, truncated at order 3, has the form:
\begin{equation}
\label{general_GH}
\left\{ 
\begin{array}{l}
\dot x = x(\lambda + a_1x^2 +a_2y^2+a_3 z^2)\\ \\
\dot y =  y(\lambda + a_1y^2 +a_2z^2+a_3x^2)\\\\
\dot z =z(\lambda + a_1z^2 +a_2x^2+a_3y^2)
\end{array}
\right.
\end{equation}
where $a_1, a_2, a_3 \in \RR$ and $\lambda\in \RR^+$.
System \eqref{general_GH} is related to the May-Leonard system \eqref{general} with $\gamma=0$, \emph{via} the change of variables:
$$
\overline{x} \mapsto x^2, \qquad \overline{y} \mapsto y^2\qquad \text{and} \qquad \overline{z} \mapsto z^2.
$$
To convert \eqref{general_GH} into \eqref{general} it is also necessary to carry out a rescaling of time and variables $a_1$,  $\lambda$, in order to set $\lambda=a_1=1$.

\begin{table}[htb]
\label{notation1}
\begin{center}
\begin{tabular}{|c|c|c|} \hline 
Variational Matrix & \quad  Eigenvalues \quad  \qquad  &\qquad   Eigenvectors \qquad \quad  \\
\hline \hline
&& \\
$Df_{0}(\pm O_1)=\left(\begin{array}{ccc} \dpt   -1 &\dpt   -1-c \dpt &e-1 \\ \\ \dpt  0 &\dpt   e & \dpt 0 \\ \\ 0 &\dpt   0 &\dpt -c  \end{array}\right)$ &  $\begin{array}{c} \dpt    e  \\ \\  -1 \\  \\ -c \\ \\  \end{array}$ & $\begin{array}{cl} \dpt   & \dpt \left(\frac{1+c}{1+e}, -1, 0 \right)  \\ &\\ \dpt   & \dpt \left(1, 0, 0 \right) \\ & \\  \dpt   &\dpt  \left(\frac{e-1}{1-c}, 0, 1 \right) \\ &\\  \end{array}$ \\
 \hline
\hline && \\
$Df_{0}(\pm O_2)=\left(\begin{array}{ccc} \dpt   -c &\dpt   0 \dpt &0 \\ \\ \dpt  e-1 &\dpt   -1 & \dpt -1-c \\ \\ 0 &\dpt   0 &\dpt e  \end{array}\right)$ &  $\begin{array}{c} \dpt    e  \\ \\  -1 \\  \\ -c \\ \\  \end{array}$ & $\begin{array}{cl} \dpt  & \dpt \left(0, \frac{1+c}{1+e}, -1 \right)  \\ &\\ \dpt   & \dpt \left(0, 1, 0 \right) \\ & \\  \dpt &\dpt  \left(1, \frac{e-1}{1-c}, 0 \right) \\ &\\  \end{array}$ \\
 \hline
\hline
&& \\
$Df_{0}(\pm O_3)=\left(\begin{array}{ccc} \dpt   e &\dpt   0 \dpt &0 \\ \\ \dpt  0 &\dpt   -c & \dpt 0 \\ \\ 1+c &\dpt   0 &\dpt -1  \end{array}\right)$ &  $\begin{array}{c} \dpt    e  \\ \\  -1 \\  \\ -c \\ \\  \end{array}$ & $\begin{array}{cl} \dpt   & \dpt \left(-1, 0, \frac{1+c}{1+e}  \right)  \\ &\\ \dpt   & \dpt \left( 0, 0, 1 \right) \\ & \\  \dpt  &\dpt  \left(0, 1, \frac{e-1}{1-c}  \right) \\ &\\  \end{array}$ \\
 \hline
\end{tabular}
\end{center}
\bigskip
\caption{\small Variational matrices, eigenvalues and eigendirections associated to the six equilibria $\pm O_1$, $\pm O_2$ and $\pm O_3$ of \eqref{general} for $\gamma=0$.}
\end{table}

\bigbreak
\subsection*{Dynamics for the May-Leonard system}
Based on \cite{GH88, Rodrigues2013}, there exists an open set of parameters satisfying \textbf{(C1a)} for which there is an invariant two-dimensional sphere which attracts all trajectories, except the origin. As illustrated in Figure \ref{GH1}, the intersection of this sphere with the axes gives rise to six saddle-type equilibria, say 
$$\pm O_1\mapsto (\pm 1, 0,0), \qquad \pm O_2\mapsto (0, \pm 1, 0)\qquad \text{and} \qquad \pm O_3\mapsto (0,0, \pm 1),$$
whose radial, contracting and expanding eigenvalues/eigendirections are described in Table 1. The formal definition of radial, contracting and expanding eigenvalue/eigendirection may be found in \cite{AC, Rodrigues2013} for instance.

The intersection of the sphere with the coordinate planes $\Fix\, \ZZ_2(k_i)$ for $i=1,2,3$, generates one-dimensional heteroclinic connections linking the equilibria. The union of these equilibria and connections forms a heteroclinic network that  will be denoted by $\Gamma$. The set $\Gamma$ is the union of eight heteroclinic cycles, each one lying on the boundary of each octant. The two-dimensional coordinate subspaces are flow-invariant and prevent visits to more than one cycle in the network. 
The parameters $e$ and $c$ of \eqref{general} have been chosen in such a way that the network is \emph{asymptotically stable}. \medbreak
The unstable manifolds of the saddles (lying in the closure of the first octant) are given by:
$$
W^u(+O_1)\subset \{(x,y,z)\in \RR^3: \quad z=0 \quad \wedge \quad x \geq 0\},
$$
$$
W^u(+O_2)\subset \{(x,y,z)\in \RR^3: \quad x=0 \quad \wedge  \quad y \geq 0\},
$$
and
$$
W^u(+O_3)\subset \{(x,y,z)\in \RR^3: \quad y=0 \quad \wedge \quad z\geq 0\}.
$$

\bigbreak
The constant $\delta= c/e >1$ measures the \emph{strength of attraction} of each cycle in the absence of perturbations. There are no periodic solutions for the case  $\delta>1$:  typical trajectories starting near $\Gamma$ (but not within $\Gamma$) approach closer and closer one of the cycles in the network and  remain near the equilibria for increasing periods of time. These trajectories make fast transitions from one equilibrium point to the next.

The network $\Gamma$ is \emph{robust} due to a biological constraint: if a species is extinct at time $t = t_0$, it will remain
extinct for all $t > t_0$, $t_0 \in \RR_0^+$.  Within each of the invariant planes defined by $x = 0$, $ {y = 0}$ and ${z = 0}$, the relevant
connecting orbit is a saddle-sink connection, and therefore the network
is \emph{structurally stable}.

In the remaining analysis, we concentrate our attention on the cycle $\Gamma_1$:
$$
\Gamma_1= \Gamma \cap \{(x, y, z)\in \RR^3: \quad x\geq 0, \quad y\geq 0, \quad z\geq 0\}.
$$

\begin{figure}
\begin{center}
\includegraphics[height=5.5cm]{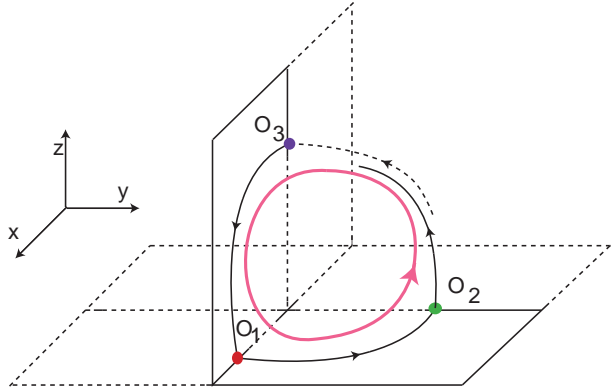}
\end{center}
\caption{\small General constant perturbations result in long-period periodic solutions that lie close
to the original cycle $\Gamma_1$.}
\label{attracting_orbit}
\end{figure}

As suggested by Figure \ref{attracting_orbit}, general symmetry-breaking constant perturbations to \eqref{general}  are well known
 to result in long-period attracting periodic solutions that lie close
to the original cycle. 
Note that, for $\gamma>0$, the planes defined by the equations $y=0$ and $z=0$ remain flow-invariant and, when restricted to the unit sphere, the periodic forcing is \emph{non-negative}. 
\subsection*{Terminology}
\label{ss:term}
For future use, we settle the following notation:
 \begin{eqnarray*}
\dpt a_1&=& \frac{c^2}{c^2+4\omega^2}  \qquad \qquad \dpt  a_2= \frac{e^2}{e^2+4\omega^2} \\ \\
b_1&=& \frac{2\,c\, \omega}{c^2+4\omega^2}  \qquad  \qquad \dpt b_2= \frac{2\, e\, \omega}{e^2+4\omega^2}  \\ \\
 \xi& =& \frac{e^2+ce +c^2}{e^3} \quad \qquad  \delta= c/e 
    \end{eqnarray*}

\section{Main result and framework of the article}
\label{s:MR}

Let $\mathcal{T}$ be a tubular neighborhood of the May-Leonard network $\Gamma$, which exists for  system \eqref{general} with $\gamma=0$.  We define a cross-section $\Sigma$ to which all trajectories in  $\mathcal{T}$ intersect transversely. For $\gamma \geq 0$, repeated intersections define a subset $\mathcal{D}\subset \Sigma$ where the return map to $\mathcal{D}$ is well defined.  Under Hypotheses \textbf{(C1a)} and \textbf{(C1b)}, we may obtain an approximation of the first return map reduced to the \emph{leading phase coordinate}.  
As well as the values of the coordinates, the non-autonomous nature of the dynamics requires us to
keep track of the \emph{elapsed time} spent on each part of the trajectory. 

From now on, let us denote by $\mathcal{F}_\gamma$ the map which comprises the \emph{leading phase coordinate} of the first return map and the \emph{elapsed time}. The detailed construction of this map, as well the topology of the approximation, may be found in Section \ref{overview}. The novelty of this article is the following result:
\bigbreak

\begin{maintheorem}
\label{Th A}
For $\omega>0$ sufficiently small, there exists $\xi^\star >0$ such that for  all $\xi>\xi^\star$ the following inequality holds:
$$
\liminf_{r\rightarrow 0^+}\, \,  \frac{ \emph{Leb} \left\{\gamma \in [0,r]: \mathcal{F}_\gamma  \text{  exhibits a strange attractor with a SRB measure}\right\}}{r}  >0.
$$
where \emph{Leb} denotes the one-dimensional Lebesgue measure.
\end{maintheorem}
\bigbreak

The existence of a set  with \emph{positive lower Lebesgue density at 0}  for which we observe strange attractors justifies  the title of this manuscript. These  strange attractors and SRB measures have strong statistical properties that will be made precise in Sections \ref{s: theory} and \ref{proof Th B}. 
The proof of Theorem \ref{Th A} is performed  in Section \ref{proof Th B} by reducing the analysis of the two-dimensional map $\mathcal{F}_\gamma$ to the dynamics of a one-dimensional map, via the \emph{Theory of rank-one attractors}.

\subsection*{Numerical evidences}
The existence of non-hyperbolic strange attractors has been suggested by the numerics presented by J. Dawes and T.-L. Tsai \cite{DT3, TD2, TD} when $\omega \approx 0$ and $\delta \gg 1$, namely:
\begin{enumerate}
\item the existence of $n$ periodic orbits, apparently for all natural numbers $n$;
\item the maximum return times of orbits appears to increase without an upper bound;
\item the test for chaos developed by Gottwald and Melbourne \cite{GM} indicates the presence of chaos;
\item the existence of non-trivial rotation intervals \cite{MT}. 
\end{enumerate}
%Note that  $\delta \gg 1 \Rightarrow \xi\gg 1$ (cf. \S \ref{s:preparatory}).

\subsection*{Structure of the paper} The rest of this article is organised as follows: in Section \ref{overview}, we state all results related to the topic and we explain how Theorem \ref{Th A} fits in the literature. The proof of this result is performed using the \emph{Theory of rank-one attractors}, whose basic ideas  are explained in Section \ref{s: theory}. In Section \ref{s:preparatory}, we refine  some results  used to prove the main theorem in Section \ref{proof Th B}. We point out some bifurcations in the family of vector fields \eqref{general} in Section \ref{s:mechanism}, emphasising the role of the parameter $\xi$.  Section \ref{s:Discussion1} concludes this article with a discussion.
Throughout this paper, we have endeavoured to make a self contained exposition bringing together all topics related to the proofs. We have drawn illustrative figures to make the paper easily readable.

\section{Overview}
\label{overview}
For completeness, we give a complete overview on the subject of the article and we explain how our result fits in the literature. The main results on the topic are summarised in Table~ 3. In what follows, we use the terminology defined at the beginning of Section \ref{s:MR}.
\medbreak
For $\gamma\geq 0$, the return map to a cross section $\Sigma$ of $f_\gamma$ depends on the phase space and on the initial time. 
The next theorem yields a description of a map that comprises the leading component of the first return map and the return time $s$ at which orbits reach the cross section:

\begin{theorem}[\cite{AHL2001, TD}, adapted]
\label{Th1}
For $\gamma \geq 0$, there is $\tilde\varepsilon>0$ (small) such that a solution of \eqref{general} that starts in $\Sigma$ at time $s$, returns to $\Sigma$ with the dynamics dominated by the coordinate $x$, defining a map on the cylinder 
$$
(x,s)\in \mathcal{D}:=\{ x/\tilde \varepsilon  \in \, \, ]\, 0, 1] \quad \text{and} \quad s\in \RR \pmod{\pi/\omega}\},
$$
that is approximated, in the $C^3$--Whitney topology, by: 
 \begin{equation}
 \label{first return2}
\mathcal{F}_\gamma(x,s)= (\mathcal{F}^1_{\gamma} (x,s) , \mathcal{F}^2_{\gamma} (x,s) \pmod{\pi/\omega} )
 \end{equation}
 with
 \begin{eqnarray*}
\mathcal{F}^1_{\gamma} (x,s) &=& \mu x^\delta + \gamma [  \mu_1 +\mu_2 (-a_1 \cos(2\omega \Phi(x,s))-b_1 \sin (2\omega \Phi(x,s))  \\ \\
&&- \mu_4(-a_1 \cos(2\omega(\mathcal{F}^2_{\gamma}(x,s) -\Delta_3))-b_1\sin(2\omega (\mathcal{F}^2_{\gamma}(x, s)-\Delta_3)))\\ \\
&& -\mu_5 (-a_2 \cos(2\omega \mathcal{F}^2_{\gamma}(x,s))-b_2\sin(2\omega \mathcal{F}^2_{\gamma}(x,s)))]+ \mathcal{O}(\gamma^2),
  \end{eqnarray*}
 \begin{eqnarray*}
\mathcal{F}^2_{\gamma} (x,s) &=& s+\mu_3-\xi \log(x)- \frac{\gamma\, \xi}{e\,  x}[\eta_\omega(x, s)-a_2\cos(2\omega s)+ b_2 \sin(2\omega s)] + \mathcal{O}(\gamma^2), \\
  \end{eqnarray*}
where
$$
\Phi(x,s)= s  + \mu_3 -\xi \log(x), \qquad 
 \eta_\omega (x, s)=\frac{e^2 \cos^2(\omega s)+2\omega^2}{ (e^2+4\omega^2)}  
  $$
and 
  $\mu_1, \mu_2, \mu_3, \mu_4, \mu_5\in \RR, \, \, \Delta_3\in \RR^+$ depend on the transition maps.
  
%  the ellipsis stand for asymptotically small  terms depending on $x$ and $y$ which converge to zero along their derivatives.
  
   \end{theorem}

The map $ \eta_\omega$ does not depend on $x$. The amplitude $\gamma$ of the periodic forcing is sufficiently small so that the $\mathcal{O}(\gamma^2)$--terms  are neglected, where $\mathcal{O}$ denotes the standard \emph{Landau notation}. The proof of Theorem \ref{Th1} is partially performed in \cite{TD} by composing local and transition maps around the equilibria. We say ``partially'' because the authors used a linearisation form which, in principle, is valid just in the $C^1$--topology. In \S \ref{ss: first return},  using results by Wang and Ott~\cite{WO}, we revisit the computation of $\mathcal{F}_\gamma$ in a $C^3$--controlled manner. %The main steps of the construction  are revived in \S \ref{ss: first return}. 

 For $\gamma>0$, we distinguish four dynamical regimes for \eqref{general},  which depend subtly on the  following parameters in the problem: the saddle-value $\delta$,  the frequency $\omega$ of the non-autonomous periodic perturbation and a constant $\mu_1$ that depends on the global parts of the dynamics. The four regimes, summarised in Table 2, are:\\
 \begin{description}
 \item[Case 1] $\delta \gtrsim 1$ and $\omega \approx 0$, \\
  \item[Case 2] $\delta \gg 1$ and $\omega \approx 0$, \\
   \item[Case 3] $\xi <{2\mu_1}$ and $\omega \gg 0$, \\
    \item[Case 4] $\xi >{2\mu_1}$  and $\omega \gg 0$.
 \end{description}
 \bigbreak
 
  \begin{table}[htb]
\begin{center}
\begin{tabular}{|c|c|c|} \hline 
{Parameters}  & \qquad  \qquad $\delta \gtrsim 1$ \qquad  \qquad &\qquad  \qquad $\delta \gg 1$ \qquad \qquad   \\
\hline \hline
&&\\
$\omega \approx 0$&Case 1 &Case 2 \\  
 & &  \\  \hline \hline \hline
 {Parameters}  & \qquad  \qquad $\xi <{2\mu_1}$ \qquad  \qquad &\qquad  \qquad $\xi >{2\mu_1}$ \qquad \qquad   \\
\hline \hline  &&\\
$\omega \gg 0$ &Case 3 & Case 4 \\
&& \\ \hline

\hline

\end{tabular}
\end{center}
\label{notationA}
\bigskip
\caption{\small Four different cases for the dynamics of \eqref{first return2}.}
\end{table} 

The meaning of the  terminology suggested by \cite{TD}  is the following. Without causing qualitative changes in the bifurcation structure:  \\
 \begin{itemize}
 \item $\mathbf{\delta \gtrsim 1}$ means that,  for a given $\gamma>0$, $\delta $ is such that $\gamma^{\delta-1}\approx 1$ and  thus the term $\mu x^\delta$ cannot be omitted from the expression of $\mathcal{F}^1_{\gamma}$; \\
  \item $\mathbf{\delta \gg 1}$ means that, for a given $\gamma>0$, $\delta$ is so large  that $\gamma^{\delta-1}\approx 0$ and  thus the term $\mu x^\delta$ may be ignored from  the expression of $\mathcal{F}^1_{\gamma}$; \\
   \item $\mathbf{\omega \approx 0}$ means that  the value of $\omega$ is so small that $b_1$ and $b_2$ may be ignored from $\mathcal{F}^1_{\gamma} $;  \\% that $b_1$ and $b_2$ may be approximated by $0$ and $a_1$ and $a_2$ may be approximated by $1$. \\
      \item $\mathbf{\omega \gg 0}$ means that the value of $\omega$ is so large that ${a_1}, a_2, b_1$ and $b_2$ may be approximated by 0 in  $\mathcal{F}^1_{\gamma}$. \\
\end{itemize}
 A clarification of this notation will be clearer in \S  \ref{ss: variables}. In what follows, we describe the expression of $\mathcal{F}_\gamma$    for each of the previous cases and the associated dynamics.

\subsection{Cases 1 and 2}

In Cases 1 and 2, we have $\omega \approx 0$, which implies that $b_1$ and $b_2$ vanish and $a_1$ and $a_2$ are close to 1 (cf. \S \ref{ss: variables}). 
For  $\delta > 1$,    $\mathcal{F}_{\gamma} $--iterates lie close to an invariant curve that may be well approximated by: 
 \begin{eqnarray}
 \label{eq1}
 \mu x^\delta + \gamma\mu_1  \left[ 1- \sqrt{a_1}\cos (2\omega \Psi(x,s))\right] 
 \end{eqnarray}
  where $\sqrt{a_1}<1$, $\Psi:  \mathcal{D} \rightarrow \RR$ is $C^3$--smooth and  $1- \sqrt{a_1}\cos (2\omega \Psi(0,s))$ is a Morse function with finitely many non-degenerate critical points. %(cf. Remark \ref{rem:non-deg}).  
  Although all the theory is valid for a more general map, we assume hereafter that:
  \bigbreak
 \begin{description}
\item[(C2)]$ \Psi(x,s)=s$.
 \label{Psi1}
 \end{description}
 \bigbreak
  Comparisons with  numerics of   \cite{DT3, TD2, TD} show that the model \eqref{eq1}, under Hypothesis \textbf{(C2)}, is sufficient to capture the dynamics. This is the reason why we assume, from now on, that these conditions are verified. %They also simplifies the computations.
Therefore, we may rewrite Theorem \ref{Th1} (with $\mu=1$) as:

\begin{proposition}
\label{first return Cases 1 and 2}
For $\gamma \geq 0$, there is $\tilde\varepsilon>0$  such that a solution of \eqref{general} that starts in $\Sigma$ at time $s$, returns to $\Sigma$ with the dynamics dominated by the coordinate $x$, defining a map on the cylinder 
$$
(x,s)\in \tilde{\mathcal{D}}:=\{ x/\tilde \varepsilon  \in \, \, ]\, 0, 1] \quad \text{and} \quad s\in \RR \pmod{1}\},
$$
that is approximated, in the $C^3$--Whitney topology, by:
$$\mathcal{F}_\gamma(x,s)= (\mathcal{F}^1_{\gamma} (x,s) \, , \,  \mathcal{F}^2_{\gamma} (x,s) )$$
where:
\begin{equation}\label{Case1B.1}
\left\{
\begin{array}{l}
\mathcal{F}^1_{\gamma} (x,s)=x^\delta + \gamma\mu_1 (1- \sqrt{a_1} \cos (2\pi s)) \\ \\
\mathcal{F}^2_{\gamma} (x,s)= \dpt s+\frac{\mu_3 \omega}{\pi} - \dpt \frac{\xi \omega}{ \pi} \log (x^\delta + \gamma  \mu_1(1- \sqrt{a_1} \cos (2\pi s))) +\mathcal{O}(\gamma) \pmod{1}\\ 
\end{array}
\right.
\end{equation}
\end{proposition}
\bigbreak

A clarification of the expression of $\mathcal{F}^2_{\gamma}$ is given in \S  \ref{P3.2}. %It will be also clear what we mean by ``\emph{approximately}''.
For $\delta>1$ and $\gamma>0$ small, the map $x\mapsto x^\delta + \gamma $ has two fixed points. In what follows, let us denote by ${x}^\star$ its \emph{positive stable fixed point}, as stressed  in Figure \ref{diagonal1}.

\begin{figure}[ht]
\begin{center}
\includegraphics[height=6cm]{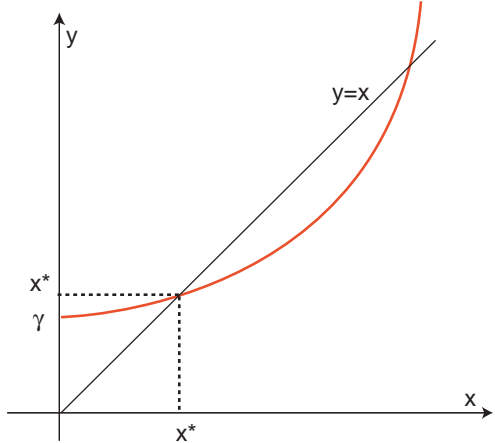}
\end{center}
\caption{\small For $\delta>1$ and $\gamma>0$, the  map $x\mapsto x^\delta + \gamma $  has two fixed points. The symbol $x^\star$ denotes the positive stable fixed point. }
\label{diagonal1}
\end{figure}

\begin{theorem}[Case 1, \cite{AHL2001, TD}, adapted]
\label{Th3.3}
If $\delta>1$ and $\gamma>0$ are such that ${x}^\star > \gamma$, there exists $\omega_0>0$ such that for all $\omega  \in \, \, ]\, 0, \omega_0\,[$, the system \eqref{Case1B.1} has an invariant closed curve as its maximal attractor.  
\end{theorem}
Under the conditions of Theorem \ref{Th3.3}, for $\omega  \in \, \, ]\, 0, \omega_0\,[$ fixed, if  $\xi$ is sufficiently large, then  the closed curve may break and saddle-node and period-doubling bifurcations may occur. 
This is implicit in \cite{TD}. The dynamics of Case 1 alternates between an invariant curve and saddle-node bifurcations, giving rise to \emph{bistability dynamics} for  \eqref{Case1B.1}: coexistence of a stable fixed point and a stable invariant curve (cf. Region III of \cite{TD2}).

\begin{theorem}[Case 2, \cite{AHL2001}, adapted] 
\label{Th3.2}
For $\gamma>0$ sufficiently small and $C>2$, if $$\frac{\exp(C/\xi\omega)-1}{ \exp(C/\xi\omega)-1/C}<\sqrt{a_1}< 1$$ then there exists a hyperbolic invariant closed set $\Lambda$ such that the dynamics of $\mathcal{F}_\gamma|_\Lambda$ is topologically conjugate to the Bernoulli shift on two symbols\footnote{ In \cite{AHL2001}, the constant $\omega$ was not a bifurcation parameter; in their case $\omega=1$ and $C=10$. The constant $C>2$ is related to the number of symbols coding the horseshoes (see \emph{rotational horseshoes} of \cite{PPS}).}. 
\end{theorem}

The set $\Lambda$ is hyperbolic and topologically transitive. Since $\mathcal{F}_\gamma$ is $C^2$, this class of objects has zero Lebesgue measure.  Theorem \ref{Th A} of this article gives a conclusive analytical result, ensuring that the flow of \eqref{Case1B.1} exhibits \emph{observable and persistent chaotic} dynamics in Case 2. When $\delta \gg 1$ ($\Rightarrow \xi \gg 1$), the dynamics is chaotic for $\omega \gtrsim0$ \cite{TD2}.

\subsection{Cases 3 and 4}

\begin{proposition}
\label{first return Cases 3 and 4}
For $\gamma \geq 0$, there is $\tilde\varepsilon>0$ such that a solution of \eqref{general} that starts in $\Sigma$ at time $s$, returns to $\Sigma$ with the dynamics dominated by the coordinate $x$, defining a map on the cylinder $ \tilde{\mathcal{D}}$  that is approximately given by 
$$\mathcal{F}_\gamma(x,s)= (\mathcal{F}^1_{\gamma} (x,s)\,  , \, \mathcal{F}^2_{\gamma} (x,s))$$
where:
\begin{equation}\label{Case4.1}
\left\{
\begin{array}{l}
\mathcal{F}^1_{\gamma} (x,s)= \gamma\, \mu_1 \\ \\
\mathcal{F}^2_{\gamma} (x,s)=\dpt s+\frac{\mu_3 \, \omega}{\pi}-\frac{\xi\, \omega}{\pi }  \log(\gamma\,  \mu_1 )- \frac{\xi \, \omega}{2\, e\,  \pi\,  \mu_1 } + \frac{\xi }{2\, \pi \,  \mu_1} \sin (2\pi s)   \pmod{1}\\ 
\end{array}
\right.
\end{equation}
\end{proposition}

%Since $\sqrt{a_2}\approx \frac{e}{2\omega}$ (see \S \ref{ss: variables}), the number $\frac{\omega \xi \sqrt{a_2}}{2\pi e} \approx \frac{\xi}{4\pi\mu_1}$ and thus, in this case, the effect of the periodic forcing on the dynamics is given by the averaged perturbation $\dpt \left(\frac{\gamma(1-x)}{2}, 0, 0\right)$ to \eqref{general}.
A clarification for the expression of $\mathcal{F}^2_{\gamma}$ is given in Subsection \ref{P3.4}.
The next result shows that the dynamics of $\mathcal{F}_\gamma$ for Cases 3 and 4 are governed by the canonical \emph{family of circle maps}~\cite{Boyland}.

\begin{theorem}[Cases 3 and 4, \cite{TD}]
If $\xi <{2\mu_1}$, then  system \eqref{Case4.1} is equivalent to the canonical family of invertible circle maps; otherwise, if $\xi >{2\mu_1}$, then system \eqref{Case4.1} is equivalent to the canonical family of non-invertible circle maps.

\end{theorem}
 \emph{Rotation numbers} (associated to a circle map) form a closed subset of $[0,1]$ and, for monotone maps, it can be shown that the rotation number is unique and independent of the choice of initial condition  $s_0\in \RR$.   Non-invertible circle maps of the type of \eqref{Case4.1} with $\xi >{2\mu_1}$ suggest more complicated dynamics since  an interval of rotation numbers may exist \cite{Boyland}.
For any rotation number within the rotation interval,  there exists an initial point $s_0$ whose orbit has that rotation number. Since orbits are periodic if and only if they have rational rotation numbers, the existence of a rotation interval implies the existence of countably many periodic orbits at that parameter value. Thus, the existence of a non-trivial rotation interval for \eqref{Case4.1} implies \emph{chaos} in the sense of \cite{MT}.

\begin{table}[htb]
\begin{center}
\begin{tabular}{|c|c|c|} \hline 
 {Parameters}  & $\delta \gtrsim 1$ & $\delta \gg 1$  \\
\hline \hline
&&\\
&Case 1 &Case 2 \\
$\omega \approx 0$ & Attracting two-torus & Hyperbolic horseshoes \cite{AHL2001}  \\  &\cite{AHL2001,TD2,  TD}  &   \textcolor{blue}{Strange attractors (New)}  \\ &&\\ \hline
\hline \hline
 {Parameters}  & \qquad  \qquad ${\xi}<{2\mu_1}$ \qquad  \qquad &\qquad  \qquad ${\xi}>{2\mu_1}$ \qquad \qquad   \\ \hline \hline
&&\\
&Case 3 & Case 4 \\
$\omega \gg 0$ &  Invertible two-torus & Non-invertible two-torus\\ & \cite{DT3, TD2,  TD} &    \cite{DT3, TD2,  TD} \\ &&\\ \hline
\hline
\end{tabular}
\end{center}
\label{table2}
\bigskip
\caption{Overview of the results in the literature and the contribution of the present article (in blue) for the equation \eqref{general}.}
\end{table}

\section{Theory of rank-one maps in two-dimensions}
\label{s: theory}

We gather in this section a collection of technical facts used repeatedly in later sections. In what follows, let us denote by $C^2(\EU^1,\RR) $ the set of $C^2$--maps from $\EU^1$ (unit circle) to $\RR$.

\subsection{Misiurewicz-type map}
\label{Misiurewicz-type map}
We say that $h\in  C^2(\EU^1,\RR) $ is a \emph{Misiurewicz map} if the following hold for some neighborhood $U$ of the critical set $C = \{x \in \EU^1 : h'(x) = 0\}$.
\bigbreak
\begin{enumerate}

\item (Outside $U$) There exist $\lambda_0 > 0$, $M_0 \in \NN$, and $0 < d_0\leq 1$ such that:
\begin{enumerate}
\item for all $m\geq M_0$, if $h^i(x)\notin U$ for all $0\leq i\leq m-1$, then $|(h^m)'(x)|\geq \exp(\lambda_0 m)$.
\item for all $m\in \NN$, if if $h^i(x)\notin U$ for all $0\leq i\leq m-1$ and $h^m(x)\in U$, then $|(h^m)'(x)|\geq d_0\, \exp(\lambda_0 m)$.
\end{enumerate}

\bigbreak

\item (Critical orbits)  For all $c\in C$ and $i >0$, $h^i(c)\notin U$.

\bigbreak

\item  (Inside $U$) 

\begin{enumerate}
\item $h''(x) \neq 0$ for all $x\in U$ and
\item for all $x\in U \backslash C$, there exists $p_0(x)\in \NN$ such that $h^i(x) \notin U$ for all $i<p_0(x)$ and $\left|\left(h^{p_0(x)}\right)'(x)\right|\geq \dpt  \frac{\exp(\lambda_0 p_0(x)/3)}{d_0}$
\end{enumerate}

\end{enumerate}

These types of maps are a slight generalization of the maps studied by Misiurewicz \cite{M}. The property \emph{``to be a Misiurewicz map''} is not an open condition, \emph{i.e.} it does not persist form small smooth perturbations. 

\subsection{Rank-one maps}
\label{rank_one}

The theory of chaotic rank-one attractors has been originated from the theory of Benedicks and Carleson on H\'enon strange attractors \cite{BC91} and 
 the development that followed by Young and Benedicks \cite{YB93}.  This theory grew and has been generalised  by  Young and Wang in a sequence of articles  \cite{WY2001, WY, WY2003, WY2008}. %It takes the theory of H\'enon-like maps ([20], [4]) as a particular case. 

\bigbreak

Let $M= [0,1] \times \EU^1$ with the usual topology. We consider the two-parametric family of maps $F_{(a,b)}: M \rightarrow M$, where $a \in[0, 1]$ and $b \in \RR$ is a scalar. Let $B_0 \subset \RR\backslash\{0\}$ with $0$ as an accumulation point\footnote{ This means that there exists a sequence of elements of $B_0$ which converges to zero. In \cite{WY2001}, $B_0$ is taken to be an interval, but the updated version of the result  just asks that $0$ accumulates values of $B_0$  (cf. \cite{WY}).}. Rank-one theory states the following:
\bigbreak

\begin{description}
\item[\text{(H1) Regularity conditions}]  
\begin{enumerate}  
\item For each $b\in B_0$, the function $(x,s,a)\mapsto F_{(a,b)}$ is at least $C^3$--smooth.
\item Each map $F_{(a,b)}$ is an embedding of $M$ into itself.
\item There exists $k\in \RR^+$ independent of $a$ and $b$ such that for all $a \in  [0,1]$, $b\in B_0$ and $(x_1, s_1), (x_2, s_2) \in M$, we have:
$$
\frac{|\det DF_{(a,b)}(x_1, s_1)|}{|\det DF_{(a,b)}(x_2, s_2)|} \leq k.
$$
\end{enumerate}

\bigbreak
\item[\text{(H2) Existence of a singular limit}] 
For $a\in [0,1]$, there exists a map $$F_{(a,0)}: M \rightarrow \EU^1 \times \{0\}$$ such that the following property holds: for every $(x, s) \in M$ and $a\in [0,1]$, we have
$$
\lim_{b \rightarrow 0} F_{(a,b)}(x,s) = F_{(a,0)}(x,s).
$$

\bigbreak
\item[\text{ (H3) $C^3$--convergence to the singular limit}]   For every choice of $a\in [0,1]$, the maps $(x,y,a) \mapsto F_{(a,b)}$  converge in the $C^3$--topology to $(x,y,a) \mapsto F_{(a,0)}$ on $M\times [0,1]$ as $b$ goes to zero.

\bigbreak
\item[\text{(H4) Existence of a sufficiently expanding map within the singular limit}] 
The\-re exists $a^\star \in [-\varepsilon, \varepsilon]$  such that $h_{a^\star}(s)= F_{(a^\star, 0)}(0,s)$ is a Misiurewicz-type map (see  Subsection \ref{Misiurewicz-type map} and Remark \ref{terminologia1}).

\bigbreak
\item[\text{(H5) Parameter transversality}] Let $  C_{a^\star}$ denote the critical set of a Misiurewicz-type map $h_{a^\star}$ (see Subsection \ref{Misiurewicz-type map}).
For each $s\in C_{a^\star}$, let $p = h_{a^\star}(x)$, and let $\widetilde{s(a)}$ and $\widetilde{p(a)}$ denote the 
continuations of $s$ and $p$, respectively, as the parameter $a$ varies around $a^\star$. The point $\widetilde{p(a)}$  is the unique point such that $\widetilde{p(a)}$  and $p$ have identical symbolic itineraries under $h_{a^\star}$ and $h_{\tilde{a}}$, respectively. We have:
$$
\frac{d}{da} h_{\tilde{a}}(s(\widetilde{a}))|_{a=a^\star} \neq \frac{d}{da} p(\tilde{a})|_{a=a^\star}.
$$

\bigbreak

\item[\text{(H6) Nondegeneracy at turns}] For each $s\in C_{a^\star}$ (set of critical points of $h_{a^\star}$), we have
$$
\frac{d}{ds} F_{(a^\star,0)}(x,s) |_{x=0} \neq 0.
$$

\bigbreak

\item[\text{(H7) Conditions for mixing}] If $J_1, \ldots, J_r$ are the intervals of monotonicity of the Misiurewicz-type map $h_{a^\star}(s)= F_{(a^\star, 0)}(0,s)$, then:
\medbreak
\begin{enumerate}
\item $\exp(\lambda_0/3)>2$ (see the meaning of $\lambda_0$ in Subsection \ref{Misiurewicz-type map}) and
\medbreak
\item if $Q=(q_{im})$ is the matrix of all possible transitions  defined by:
\begin{equation*}
\left\{
\begin{array}{l}
1 \qquad \text{if} \qquad J_m\subset h_{a^\star} (J_i)\\  
0 \qquad \text{otherwise},\\
\end{array}
\right.
\end{equation*}
then there exists $N\in \NN$ such that $Q^N>0$ (\emph{i.e.} all entries of the matrix $Q^N$, endowed with the usual product,  are positive).
\bigbreak
\end{enumerate}

\end{description}

\begin{remark}
\label{terminologia1}
Identifying $  \{0\}\times \EU^1$ with $\EU^1$, we refer to $F_{(a,0)}$  the circle map $h_a : \EU^1 \rightarrow \EU^1$ defined by $h_a(s) = F_{(a,0)}(0,s)$ as the \emph{singular limit} of $F_{(a,b)}$.
\end{remark}

\subsection{Strange attractors and SRB measures}
Following \cite{WY},
we formalize the notion of strange attractor supporting an ergodic SRB measure, for a two-parametric family $F_{(a,b)}$ defined on the set  $M= [0,1]\times \EU^1$, endowed with the induced topology. In what follows, if $A\subset M$, let us denote by $\overline{A}$ its \emph{topological closure}. 

\medbreak
Let $F_{(a,b)}$ be an embedding  such that $F_{(a,b)} (\overline{U} )\subset U$ for some open set $U\subset M$. In the present work we refer to 
$$
{\Omega} =\bigcap_{m=0}^{+\infty}  F_{({a},b)}^m(\overline{U}).
$$
as an \emph{attractor} and $U$ as its \emph{basin}. The attractor $\Omega$ is \emph{irreducible} if it cannot be written as the union of two (or more) disjoint attractors.
\medbreak
%We say that $\Omega$ is a \emph{strange attractor} for $F_{({a},b)}$ if, Lebesgue almost all points in the basin of attraction of $\Omega$ are Lyapunov-unstable.   

\bigbreak
\begin{definition}
We say that $ F_{({a},b)}$ possesses a \emph{strange attractor supporting an ergodic SRB measure} $\nu$ if:
\begin{itemize}
\item for Lebesgue almost all $(x,s) \in U$, the $ F_{({a},b)}$--orbit of $(x,s)$ has a positive
Lyapunov exponent, \emph{i.e.}
$$
\lim_{n \in \NN} \frac{1}{n} \|D  F_{({a},b)} ^n(x,s) \|>0;
$$

\item $ F_{({a},b)}$ admits a unique ergodic SRB measure (with no-zero Lyapunov exponents);

%\item the conditional measures of $\nu$ on unstable manifolds are equivalent to the Riemannian volume on these leaves and 
\item  for Lebesgue almost all points $(x,s)\in U$ and  for every continuous function $\varphi : U\rightarrow \RR$, we have:

\begin{equation}
\label{limit2}
\lim_{n\in \NN} \quad \frac{1}{n} \sum_{i=0}^{n-1} \varphi \circ  F_{({a},b)}^i(x,s) = \int  \varphi \, d\nu.
\end{equation}

\end{itemize}
\end{definition}

Admitting that $F_{({a},b)}$ admits a unique ergodic SRB measure $\nu$, we define convergence of  $F_{({a},b)}$ with respect to $\nu$.

\begin{definition}
We say that:
\begin{itemize}
\item $F_{({a},b)}$ converges  (in distribution with respect to $\nu$) to the normal distribution if, for every Holder continuous function $\varphi: U \rightarrow \RR$, the sequence $\left\{\varphi\left( F_{(\tilde{a},b)}^i\right) : i\in \NN \right\}$ obeys a \emph{central limit theorem}; in other words, if $\int \varphi \, d \nu = 0$, then the sequence
$$
\frac{1}{\sqrt{m}} \sum_{i=0}^{m-1}\varphi \circ  F_{(\tilde{a},b)}^i
$$
converges in distribution (with respect to $\nu$) to the \emph{normal distribution}.\\
\item the pair $(F_{({a},b)}, \nu)$ is \emph{mixing} if it is isomorphic to a Bernoulli shift. 

\end{itemize}
\end{definition}
%The variance of the limiting normal distribution is strictly positive unless $$\varphi \circ  F_{(\tilde{a},b)} = \Psi \circ  F_{(\tilde{a},b)} -\Psi, \qquad \text{for some}\qquad  \Psi \in L^2(\nu).$$

%Our analysis is based on \cite{WY2001, WY, WY2003}, which together contain one of the very few rigorous theories of rank-one strange attractors.
%We discuss briefly below the approach of \cite{WY2001, WY, WY2003} and how to apply the results.

\subsection{Q. Wang and L.-S. Young's reduction}
The results developed in \cite{WY2001, WY, WY2003, WY2008} are about maps with attracting sets on which there is strong dissipation
and (in most places) a single direction of instability.  Two-parameter families
${F_{(a,b)}}$ have been considered  and it has been proved that if a singular limit makes sense (for $b=0$) and if the resulting family of 1D maps has certain
``good'' properties, then some of them can be passed back to the two-dimensional system ($b > 0$). They  allow us to prove results on strange attractors for a positive Lebesgue measure set
of $a$.

\medbreak
For attractors with strong   dissipation and one direction of instability, Wang and Young conditions \textbf{(H1)--(H7)} are  simple and checkable; when satisfied, they guarantee the existence of  strange attractors with a \emph{package of statistical and geometric properties}:

\begin{theorem}[\cite{WY}, adapted]
\label{th_review}
Suppose that the two-parametric family of maps $F_{(a,b)}$ satisfies \textbf{(H1)--(H7)}. Then, for all sufficiently small $b\in  B_0$, there exists a subset $\Delta \in  [-\varepsilon, \varepsilon]$ with positive Lebesgue measure such that for all $a\in \Delta$, the map $F_{({a},b)}$ admits an irreducible strange attractor $\tilde{\Omega}\subset \Omega$  that supports a unique ergodic SRB measure $\nu$. The orbit of Lebesgue almost all points in $\tilde{\Omega} $ is asymptotically distributed according to $\nu$. Furthermore,  $(F_{({a},b)}, \nu)$ is \emph{mixing}.
\end{theorem}

In contrast to earlier results, the theory in \cite{WY2001, WY, WY2003, WY2008} is
generic, in the sense that the conditions under which it holds rely only to certain
general characteristics of the maps and not to specific formulas or contexts. 
% The purpose of the present paper is to fit the Bykov's bifurcation scenario into the framework of \cite{WY2001, WY, WY2003}.

\section{Preparatory section}
\label{s:preparatory}

In this section, we put together some preliminaries needed for the proof of Theorem \ref{Th A}. We start by listing some properties of the constants which appear in the expression of $\mathcal{F}_{\gamma}$. Then, in \S \ref{ss: first return}, we refine the expression of $\mathcal{F}_{\gamma}$ in the four cases under consideration. 
%First of all note that If $\delta \gtrsim 1$, then $c\gtrsim e$ and thus $a_1\approx a_2$ and  $b_1\approx b_2$.
\subsection{Control of variables}
\label{ss: variables}
In this section, we present additional information about the constants which appear in the expression of $\mathcal{F}_{\gamma}$. We set  $a_1, a_2,  b_1, b_2$ as real valued functions on  $\omega \geq 0$.
\begin{lemma}
\label{lemma1}
The following equalities are valid: \medbreak
\begin{enumerate}
\item $a_1^2+b_1^2 =a_1$ and $a_2^2+b_2^2 =a_2$. \\
\item $\dpt \lim_{\omega \rightarrow 0}  a_1(\omega)=\lim_{\omega \rightarrow 0}  a_2(\omega)=1$ \, \emph{and} \,  $\dpt \lim_{\omega \rightarrow 0}  b_1(\omega)=\lim_{\omega \rightarrow 0}  b_2(\omega)=0$. \\
\item  $\dpt \lim_{\omega \rightarrow +\infty}  a_1(\omega)= \lim_{\omega \rightarrow +\infty}  a_2(\omega)= \lim_{\omega \rightarrow +\infty}  b_1(\omega)= \lim_{\omega \rightarrow +\infty}  b_2(\omega)=0$. \\
\item $\dpt c\, \xi =\delta(1+\delta+\delta^2) $ \, \emph{and} \,  $\dpt e\, \xi = (1+\delta+\delta^2)$. \\ 
\item $\dpt \lim_{\omega \rightarrow +\infty} \left[\sqrt{a_2(\omega)}\big/ \frac{e}{2\omega}\right]=1 $  \, \emph{and} \, $\dpt \lim_{\omega \rightarrow +\infty} \left[\sqrt{a_2(\omega)}- \frac{e}{2\omega}\right]=0 $. \\
\item $\dpt \lim_{\omega \rightarrow +\infty} \left[b_2(\omega)/a_2 (\omega)\right]=+\infty $. 
\end{enumerate}

\end{lemma}

\begin{proof}
Taking into account the constants list  defined on  \S \ref{ss:term}, the proof of this result is quite elementary. For the sake of completeness, we present the proof without deep details. \medbreak
\begin{enumerate}
\item The equalities follow from:  
$$
a_1^2+b_1^2 = \frac{c^4}{(c^2+4\omega^2)^2}  +\frac{4c^2\omega^2}{(c^2+4\omega^2)^2}  =\frac{c^2(c^2+4\omega^2)}{(c^2+4\omega^2)^2}=\frac{c^2}{(c^2+4\omega^2)}  = a_1
$$
and
$$
a_2^2+b_2^2 = \frac{e^4}{(e^2+4\omega^2)^2}  +\frac{4e^2\omega^2}{(e^2+4\omega^2)^2}  =\frac{e^2(e^2+4\omega^2)}{(e^2+4\omega^2)^2}=\frac{e^2}{(e^2+4\omega^2)}  = a_2.
$$
\item It is straightforward to conclude that:
$$
\lim_{\omega \rightarrow 0}  a_1(\omega)= \lim_{\omega \rightarrow 0}  \frac{c^2}{c^2+4\omega^2}=\lim_{\omega \rightarrow 0}  \frac{e^2}{e^2+4\omega^2}=   \lim_{\omega \rightarrow 0}  a_2(\omega)=1
$$
and
$$ \lim_{\omega \rightarrow 0}  b_1(\omega) =\lim_{\omega \rightarrow 0}  \frac{2c\omega}{c^2+4\omega^2} =\lim_{\omega \rightarrow 0}  \frac{2e\omega}{e^2+4\omega^2}  =\lim_{\omega \rightarrow 0}  b_2(\omega)=0.
$$
\item By definition, it is clear that:
$$
\lim_{\omega \rightarrow +\infty}  a_1(\omega)= \lim_{\omega \rightarrow +\infty}  \frac{c^2}{c^2+4\omega^2} = \lim_{\omega \rightarrow +\infty}  \frac{e^2}{e^2+4\omega^2} = \lim_{\omega \rightarrow +\infty}  a_2(\omega) =0 
$$
and
$$
\lim_{\omega \rightarrow +\infty}  b_1(\omega) = \lim_{\omega \rightarrow +\infty}  \frac{2c\omega}{c^2+4\omega^2} = \lim_{\omega \rightarrow +\infty}  \frac{2e\omega}{e^2+4\omega^2} = \lim_{\omega \rightarrow +\infty}  b_2 (\omega)=0.
$$
\item This item follows from the equalities:
$$
c\xi =   \frac{c e^2+c^2 e +c^3}{e^3}=  \frac{c}{e} \left(\frac{e^2}{e^2} + \frac{c}{e} + \frac{c^2}{e^2}\right)=\delta(1+\delta+\delta^2) 
$$
and
$$
e\xi =   \frac{e^2+c e +c^2}{e^2}= 1+ \delta+ \delta^2.
$$

\item The item follows from:

$$
\lim_{\omega \rightarrow +\infty} \left[\sqrt{a_2(\omega)}\big/ \frac{e}{2\omega}\right] = \lim_{\omega \rightarrow +\infty} \left[{\frac{2e\omega}{e\, \sqrt{e^2+4\omega^2}}}\right]  =  
\lim_{\omega \rightarrow +\infty} \left[{\frac{2}{ \sqrt{e^2/\omega^2+4}}}\right]  = 1
$$

and
$$
\lim_{\omega \rightarrow +\infty} \left[\sqrt{a_2(\omega)}- \frac{e}{2\omega}\right] = \lim_{\omega \rightarrow +\infty} \left[\sqrt{\frac{e^2}{e^2+4\omega^2}}- \frac{e}{2\omega}\right]  =  \lim_{\omega \rightarrow +\infty} \left[\sqrt{\frac{e^2/\omega^2}{e^2/\omega^2+4}}- \frac{e}{2\omega}\right] 
=0.
$$

\item It is immediate to check that:
$$
\dpt \lim_{\omega \rightarrow +\infty}\left[b_2(\omega)/a_2 (\omega)\right] = \lim_{\omega \rightarrow +\infty}\frac{2\, e\, \omega (e^2+4\, \omega^2)}{e^2\,   (e^2+4\, \omega^2)} = \lim_{\omega \rightarrow +\infty}\frac{2\, \omega}{e}=+\infty.
$$

\end{enumerate}

\end{proof}

\subsection{First return map}
\label{ss: first return}

In this subsection, we give an expression for the first return map to a given cross section to $\Gamma$ (at leading order), obtained  as the composition of two types of maps: \emph{local maps}  between the neighbourhood walls of the hyperbolic equilibria where we may compute a normal form, and \emph{global maps} from one neighbourhood wall to another.

%The perturbing term of \eqref{general} acts on the \emph{radial direction} of the equilibria $(\pm 1, 0, 0)$ and is of the type  $\gamma(1-x) g (2\omega t)$ where $g$ is a $\pi/\omega$--periodic non-negative  map. According to by \cite{TD}, the map $\mathcal{F}_\gamma$ is the composition of local and global maps. 
%Using\cite{WO}, we prove that the approximation performed in  \cite{TD} may be taken in the $C^3$-topology. 
%We just sketch the main ideas of the proof; the interested reader may find all details in \S 2 of \cite{AHL2001} and Appendix of \cite{TD}.

%Converting the system \eqref{general} into an autonomous differential equation on the extended phase space $\RR^3 \times \EU^1$ gives rise to:
%\begin{equation}
%\label{general_aut}
%\left\{ 
%\begin{array}{l}
%\dot x = x((1-r)-cy+ez) +\gamma(1-x)\sin^2(\theta)\\ \\
%\dot y =  y((1-r)-cz+ex)\\\\
%\dot z = z((1-r)-cx+ey)\\ \\
%\dot \theta = 2\omega 
%\end{array}
%\right.
%\end{equation}
%The following analysis will be formed by $O_1$. The study of the cases $O_2$ and $O_3$ is analogous. 
%For $\gamma=0$, each equilibrium $O_i$ of \eqref{general}  gives rise to a hyperbolic periodic solution of \eqref{general_aut}. 
\subsubsection{Local map}
\label{ss: local map}
In what follows, we set $i\in\{1, 2, 3\}$ and $j\in\{s,r, u\}$. 
 Assuming Hypotheses \textbf{(C1a)} and \textbf{(C1b)}, according to Wang and Ott \cite{WO}, there exists $\gamma_i>0$  and $V_i$, a  $\tilde\varepsilon_i$--cubic neighbourhood  of $O_i$, where the vector field $f_\gamma$, $\gamma \in [0, \gamma_i]$ may be written as:
\begin{equation}
\label{general_linear1}
\left\{ 
\begin{array}{l}
\dot x_s = -c +\gamma\, \,  g_s(\gamma x_s, \gamma x_r,  \gamma x_u, \theta; \gamma) x_s\\ \\
\dot x_r = -1 +\gamma \, \, g_r(\gamma x_s, \gamma x_r,  \gamma x_u, \theta; \gamma) x_r\\ \\
\dot x_u = e + \gamma  \, \, g_u(\gamma x_s, \gamma x_r, \gamma x_u, \theta; \gamma)x_u\\ \\
\dot \theta = 2\omega 
\end{array}
\right.
\end{equation}
where
\begin{itemize}
\item $\tilde\varepsilon_i>0$ is small,  \\
\item $g_s, g_r,  g_u$ are analytic maps in $V_i\times \EU^1\times [0, \gamma_0]$, \\
\item $\|g_j\|_{C^3} <K_1$ for some  $K_1>0$ ($\|\star\|_{C^3}$ denotes the $C^3$--norm). 
\end{itemize}
 
 \begin{remark}
 The terminology $s, r, u$ corresponds to the contracting, radial and expanding direction near each equilibrium. 
The maps $g_s, g_r$ and $g_u$ depend on the equilibrium around which the normal form is being calculated. 
 \end{remark}

After rescaling variables 
$$
x_s\mapsto x_s/\tilde\varepsilon_i, \qquad x_r\mapsto x_r/\tilde\varepsilon_i, \qquad x_u\mapsto x_u/\tilde\varepsilon_i,
$$
we may define the cross sections $\In(O_i)$ and $\Out(O_i)$ as follows (see Figure \ref{CS1}):

\begin{eqnarray*}
\In(O_1) &=& \{(x_s, x_r, x_u): |x_r|\leq 1, 0\leq x_u\leq 1,  x_s=1\},\\ \\
\Out(O_1)&=&\{(x_s, x_r, x_u): |x_r|\leq 1,  x_u=1,   0\leq x_s\leq 1\}, \\\\
\In(O_2) &=& \{(x_s, x_r, x_u):x_s=1, 0\leq |x_r|\leq 1,  0\leq x_u\leq1\},\\ \\
\Out(O_2)&=&\{(x_s, x_r, x_u): 0\leq x_s\leq 1,  |x_r|\leq 1,   x_u=1\}, \\\\
\In(O_3) &=& \{(x_s, x_r, x_u): 0\leq x_u\leq 1, x_s=1 ,  |x_r|\leq1\},\\ \\
\Out(O_3)&=&\{(x_s, x_r, x_u): x_u=1,  0\leq x_s\leq 1,   |x_r|\leq 1\}. \\
 \end{eqnarray*}

\begin{figure}[ht]
\begin{center}
\includegraphics[height=6cm]{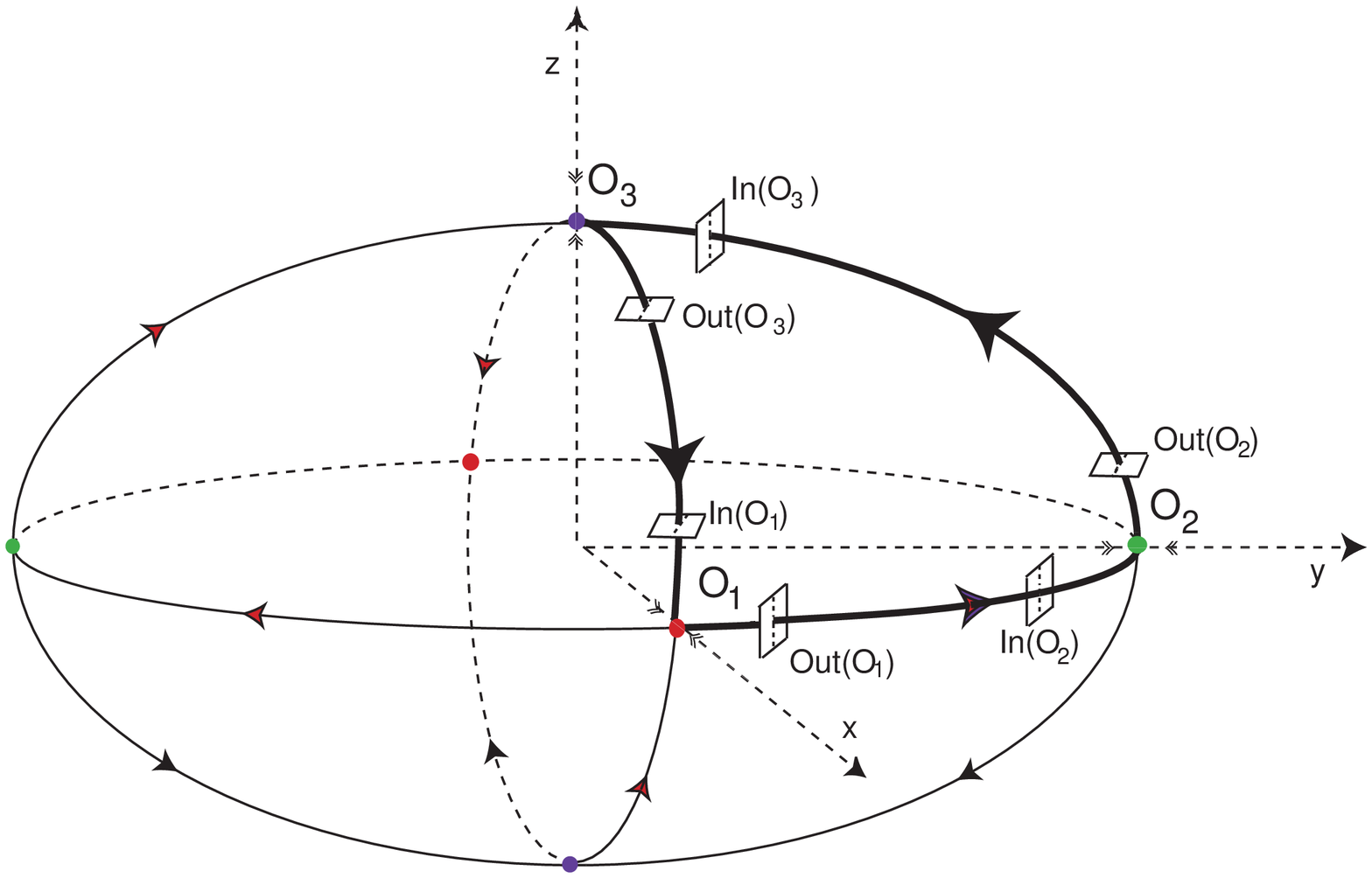}
\end{center}
\caption{\small Cross sections near the hyperbolic equilibria of the May and Leonard cycle $\Gamma_1$ (restriction of $\Gamma$ to the first octant). }
\label{CS1}
\end{figure}

The set $\In(O_i)$ contains initial conditions where the points \emph{go inside} $V_i$ in positive time; the set $\Out(O_i)$ contains initial conditions that \emph{leave} $V_i$ in positive time.   
\bigbreak
For $\gamma \in [0, \gamma_i]$, let $q_0=(x_s(0), \, x_r(0), \, x_u(0),\,  \theta(0) )\in \In(O_i)\times \EU^1 $ and let 
$$
\left\{ 
\begin{array}{l}
q(t, q_0; \gamma)=(x_s(t, q_0; \gamma), \quad x_r(t, q_0; \gamma), \quad x_u(t, q_0; \gamma),\quad  \theta(t, q_0; \gamma))\\ \\
q(0, q_0; \gamma)=q_0.
\end{array}
\right.
$$

Integrating \eqref{general_linear1}, we get:

\begin{equation}
\label{general_linear2}
\left\{ 
\begin{array}{l}
 x_s (t, q_0; \gamma)= x_s(0)\,  \exp(t\, (-c+w_s(t, q_0 ; \gamma)) \\ \\
  x_r(t, q_0; \gamma) = x_r(0)\,  \exp(t\, (-1+w_r(t, q_0 ; \gamma)) \\ \\
 x_u(t, q_0; \gamma) = x_u(0)\,  \exp(t\, (e+w_u(t, q_0 ; \gamma)) \\ \\
\theta (t, q_0; \gamma) = \theta(0)+2\omega \, t 
\end{array}
\right.
\end{equation}
where 
$$
w_j(t, q_0; \gamma)= \frac{1}{t}\int_0^t \gamma \, g_j(q(s, q_0; \gamma); \gamma) \, ds.
$$
Proposition 5.5 of \cite{WO} asserts that there exists $K_2>0$ such that the following statement holds:  for any $T^\star>1$ such that all solutions of \eqref{general_linear1} that start in $\In(O_i)\times \EU^1$ remain in $V_i\times \EU^1$ up to time $T^\star$, we have: $$\|w_j \|_{C^3}< \gamma \, K_2.$$ As well as the values of the coordinates, the \emph{non-autonomous} nature of the dynamics require us to keep track of the \emph{elapsed time} spent on each part of the solution between the  cross sections $\In(O_i)$ and $\Out(O_i)$.

Let $q_0 \in  \In(O_i)\times \EU^1$, $\gamma\in\,  ]\, 0, \gamma_i\,[$ and let $T(q_0; \gamma)$ be the time at which the solutions of \eqref{general_linear2} starting from $q_0$ reaches $\Out(O_i)$. The time of flight is defined by:
$$
T(q_0; \gamma)= \frac{1}{e + w_u(T(q_0, \gamma), q_0; \gamma)}\ln \left(\frac{1}{\gamma\,  x_u (0)}\right).
$$
Proposition 5.7 and Lemma 7.4  of  \cite{WO}  provide a precise $C^3$--control of $T$. Check also the last paragraph of \cite[\S 7]{WO}. %More specially, they ensure the existence of $K_3>0$ such that the map $T$ satisfies $$\left\|T-\frac{1}{e}\ln \left(\frac{1}{\gamma}\right)\right\|_{C^3}<K_3.$$

\subsubsection{Global map}
The derivation of the Poincar\'e map also involves the calculation of the global maps between cross-sections. According to \cite{TD}, for  $i=1,2,3$, the global maps $$\Out(O_i) \rightarrow \In \left(O_{(i+1)\text{mod}\, 3}\right)$$ are invertible linear maps. The time that elapses during the transition between these cross-sections will be denoted by $\Delta_i\geq0$.  See formulas (16), (19), (21) and (24) of \cite{AHL2001}.

\subsubsection{The return map}
We are interested in an expression for the first return map from  $\In(O_3)\backslash W^s(O_3)$ to itself\footnote{Although $O_3$ is no longer an equilibrium for $\gamma>0$, the set $\In(O_3)$ is still a cross-section.}. From now on, as suggested in Figure \ref{neigh1}, we consider just two coordinates of $\In(O_3)$:
$$
(x, z) \mapsto (x_u, 1, x_r) \in \In(O_3).
$$
  In the region of the space of parameters defined by \textbf{(C1a)}, the degree of the variables in the first phase coordinate ($x$) is smaller than that of the second ($z$).  

 The authors of \cite{TD}  proved that, for $\gamma \geq 0$, there is $\tilde\varepsilon>0$ (small) such that a solution of \eqref{general} that starts in $\Sigma= \In (O_3)\backslash W^s(O_3)$ at time $s$, returns to $\Sigma$, with the dynamics dominated by the coordinate $x$, defining a map on the cylinder 
$$
(x,s)\in \mathcal{D}:=\{ x/\tilde \varepsilon  \in\, \,  ]\, 0, 1] \quad \text{and} \quad s\in \RR \pmod{\pi/\omega }\},
$$
that may be approximated by: 
%\begin{proposition}[\cite{TD}]
%\label{first_return_general}
\begin{figure}
\begin{center}
\includegraphics[height=7cm]{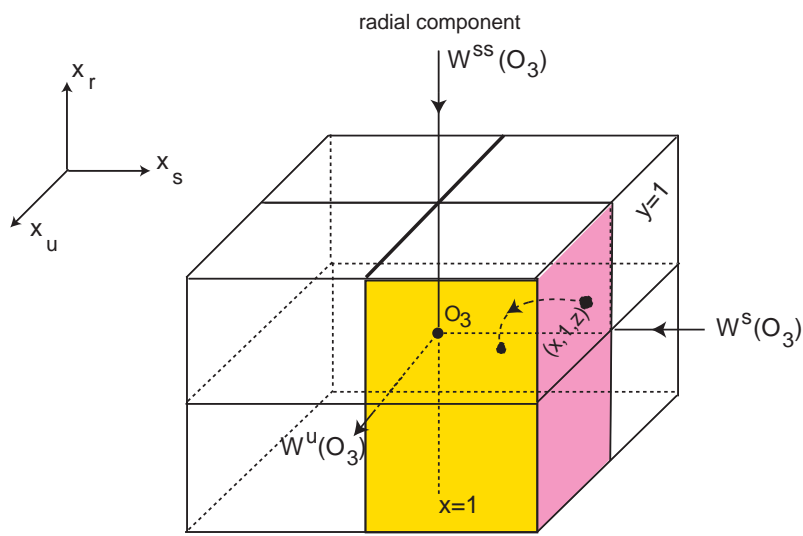}
\end{center}
\caption{\small Scheme of the cross sections ${\In(O_3)}$ [pink] and $\Out(O_3)$ [yellow] in the neighbourhood $V_3$ of $O_3$. The set $W^{ss}(O_3)$ corresponds to the local strong stable manifold of $O_3$;  it corresponds to the radial direction. }
\label{neigh1}
\end{figure}
 \begin{eqnarray*}
\mathcal{F}^1_{\gamma} (x,s) &=& \mu x^\delta + \gamma \left[  \mu_1 +\mu_2 L_1(x,s) -\mu_4 G_1(x,s) -\mu_5 G_2(x,s)\right] +O(\gamma^2) \\
\mathcal{F}^2_{\gamma} (x,s) &=& s+\mu_3-\xi \log(x)-\frac{\gamma \, \xi}{x}L_2(x,s) + O(\gamma^2) \quad \pmod{\pi/\omega} 
  \end{eqnarray*}
where $\mu_1$, $\mu_2$, $\mu_3$, $\mu_4$ and $\mu_5$ are real constants (in general they are not computable analytically) that depend on the global maps,  %$h=1=B_{110}=B_{220}=B_1$ 
and
 \begin{eqnarray*}
L_1(x,s)&=&\exp({-c\, (T_3(0)-T_2(0)-\Delta_2)}) \int_{T_2(0)+\Delta_3}^{T_3(0)} \exp(c(\tau-T_2(0) -\Delta_2) g(2\omega \tau) \, d\tau \\ \\
L_2(x,s)&=& \int_{t}^{T_1(0)} \exp(-e (\tau-t)) g(2\omega \tau) \, d\tau
   \end{eqnarray*}
 \begin{eqnarray*}
G_1(x,s)&=& \frac{1}{\exp(c\pi/\omega)-1}  \int_{0}^{\pi/\omega} \exp(c\tau) g(2\omega (T_3(\gamma)+\tau)) \, d\tau \\
 G_2(x,s)&=& \frac{1}{\exp(-e\pi/\omega)-1}  \int_{0}^{\pi/\omega} \exp(-e\tau) g(2\omega (T_3(\gamma)+\Delta_3+\tau)) \, d\tau \\
    \end{eqnarray*}
 \begin{eqnarray*}
 T_1(0)&=& s -\frac{1}{e} \log(x) \\
 T_2(0)&=& s +\Delta_1 - \frac{e+c}{e^2} \log(x) \\
  T_3(\gamma)&=& s +\Delta_1 +\Delta_2 - \frac{e^2+ce+c^2}{e^3} \log(x)  \\ && -\gamma \left[ \frac{e^2+ce+c^2}{e^3 x} \times \int_s^{T_1(0)} \exp(-e(\tau-s) g(2\omega \tau) \, d\tau)\right] + O(\gamma^2). \\
   \end{eqnarray*}
The expression of the first return map given in Theorem \ref{Th1} may be now derived, assuming $g(2 \omega t)=\sin^2(\omega \, t)$. 
Refining the expression of $L_2$, we get:

\begin{eqnarray*}
L_2(x,s) &=&\int_{s}^{T_1(0)} \exp(-e (\tau-s)) \sin^2(\omega \tau) \, d\tau \\ \\
&=& \frac{e^2+2\omega^2-e^2\cos^2(\omega \, s) +2\omega e \sin (\omega s)\cos(\omega s)}{e\, (e^2+4\omega^2)} \\ \\ 
&=& \frac{e^2 \cos^2(\omega s)+2\omega^2}{e\, (e^2+4\omega^2)} +  \frac{e^2(1-2\cos^2(\omega s))}{e\, (e^2+4\omega^2)} +  \frac{\omega \, e \sin (2\omega s)}{e\, (e^2+4\omega^2)}   \\ \\
&=&\frac{1}{e}\left( 
\eta_\omega(x,s)- a_2 \cos (2\omega s) + b_2 \cos (2\omega s) \right)
\end{eqnarray*}
where
$$
\eta_\omega (x, s):= \frac{e^2 \cos^2(\omega s)+2\omega^2}{ (e^2+4\omega^2)}.
$$
The map  $\eta_\omega$ does not depend on $x$ and
$$
\lim_{\omega \rightarrow 0}\eta_\omega (x, s) =1\qquad \text{and}\qquad \lim_{\omega \rightarrow +\infty}\eta_\omega (x, s)=\frac{1}{2},
$$
making a complete agreement between our analysis  and  that of \cite{TD}.

\subsection{Digestive remarks}
For $\tilde\gamma=\min\{\gamma_1, \gamma_2, \gamma_3\}$, we offer the following remarks concerning the expression of $\mathcal{F}_\gamma$,  $\gamma \in [0, \tilde\gamma]$.
\medbreak
 \begin{remark}
 When $\gamma = 0$, for $\mu=1$ and $x>0$, we may write:
 $$
 \mathcal{F}_{0}(x,s)= (  x^\delta, \, \, s + \omega\mu_3/\pi -\omega \xi \log(x)/\pi \pmod{1} ). $$
 This means that the $x$-component is contracting ($\delta > 1$) and thus the dynamics of $ \mathcal{F}_{0}$ is governed by the $s$-component. This is consistent to the fact that the network $\Gamma$ is \emph{asymptotically stable}. In particular, 
 $
 \dpt \lim_{x\rightarrow 0}  \mathcal{F}_{0}(x,s) = (0, +\infty).
 $
 Note that the computations of Section \ref{proof Th B} do not hold for $\gamma=0$ due to the change of coordinates \eqref{change2}.
 \end{remark}
 
\begin{remark}
\label{approx C3}
The authors of  \cite{TD}  have provided an approximation of $\mathcal{F}_{\gamma}$ in the $C^1$--norm since they   used a linearisation form.  Assuming that $c$ and $e$ satisfy the  Diophantine nonresonance Condition \textbf{(C2)}, using the arguments of  \cite{WO} revisited in \S \ref{ss: local map}, the approximation $\mathcal{F}_\gamma$ holds in the $C^3$--topology. 
\end{remark}

%\begin{remark}
%\label{rem:generalization1}

%\end{remark}
\subsection{Proof of Proposition \ref{first return Cases 1 and 2}}
\label{P3.2}
Using \eqref{eq1},  Hypothesis \textbf{(C2)} and the identity
\begin{equation}
\label{lemma_cos}
\forall y\in \RR, \quad\forall x\in \RR\backslash \{0\}, \quad \exists a\in [0,2\pi[:  \qquad x\cos a + y\sin a = \sqrt{x^2+y^2\, } \, \cos \left(\arctan \left(\frac{y}{x}\right)+a\right),
\end{equation}
the expression of $\mathcal{F}_{\gamma}$ may be rewritten as (see Lemma \ref{lemma1}):
\begin{equation}\label{Case1}
\left\{
\begin{array}{l}
\mathcal{F}^1_{\gamma} (x,s)=x^\delta + \gamma \mu_1 (1- \sqrt{a_1} \cos (2\omega s)) \\ \\
\mathcal{F}^2_{\gamma} (x,s)=\dpt s+\mu_3-\xi \log (x)- \dpt\frac{ \xi \gamma}{ e} \frac{(1-\sqrt{a_2} \cos(2\omega s))}{x}  \pmod{\pi/\omega}.\\ 
\end{array}
\right.
\end{equation}
Taking into account that $\dpt \lim_{\omega \rightarrow 0} a_1(\omega)= \lim_{\omega \rightarrow 0} a_2(\omega)$ and making the   change of coordinates $$\dpt s_{\,\text{new}}\mapsto \frac{\omega}{\pi} \, \, s_{\, \text{old}},$$ we get:
\begin{equation}\label{Case1B}
\left\{
\begin{array}{l}
\mathcal{F}^1_{\gamma} (x,s)=x^\delta + \gamma\mu_1 (1- \sqrt{a_1} \cos (2\pi s)) \\ \\
\mathcal{F}^2_{\gamma} (x,s)=\dpt s+\frac{\mu_3 \omega}{\pi} - \frac{\xi \omega}{\pi}\log (x^\delta + \gamma  \mu_1(1- \sqrt{a_1} \cos (2\pi s))) +O(\gamma^2) \,  \pmod{1}\\ 
\end{array}
\right.
\end{equation}
The $\gamma^2$--terms that appear in $\mathcal{F}^2_{\gamma} (x,s)$ arise in the Taylor $\gamma$--expansion of $$\log (x^\delta + \gamma  \mu_1(1- \sqrt{a_1} \cos (2\pi s)))$$ and hence they are collected in the expression $\mathcal{O}(\gamma^2) $.
Afraimovich \emph{et al} \cite{AHL2001} refer to equations \eqref{Case1B} as a \emph{dissipative separatrix map} since it corresponds  to the  return maps near separatrices in perturbed Hamiltonian systems when $\delta =1$. The \emph{resonant case} $\delta =1$ has also been studied in \cite{PD2005}.

\subsection{Proof of Proposition \ref{first return Cases 3 and 4}}
\label{P3.4}
Taking into account that  $$
 \dpt \lim_{\omega \rightarrow +\infty}\eta_\omega (x, s)=\frac{1}{2}, \qquad    \lim_{\omega \rightarrow +\infty} \arctan\left(\frac{b_2(\omega)}{a_2(\omega)}\right)=\frac{\pi}{2},
  $$
   using equality  \eqref{lemma_cos}, we get:
 \begin{eqnarray*}
\mathcal{F}^2_{\gamma} (x,s) &=& s+\mu_3 -\xi \log (x) - \frac{ \xi \gamma}{\, e} \frac{(1/2-\sqrt{a_2} \cos(2\omega s-\pi/2))}{x}+O(\gamma) \pmod{\omega/\pi}.
  \end{eqnarray*} 
With $\mu=1$,  since $\dpt \lim_{\omega \rightarrow +\infty } a_1(\omega)=0$, we may write:
\begin{equation}\label{Case4}
\left\{
\begin{array}{l}
\mathcal{F}^1_{\gamma} (x,s)= x^\delta + \gamma\mu_1\\ \\
\mathcal{F}^2_{\gamma} (x,s)=\dpt s+\frac{\omega \mu_3}{\pi}-\frac{\xi\omega }{\pi} \log( x^\delta + \gamma\mu_1)- \frac{\gamma \xi \omega}{2\, e\,  \pi\,   (x^\delta + \gamma\mu_1)} + \frac{ \gamma \, \omega \, \xi \sqrt{a_2}}{ \, \pi\,  e\, (x^\delta + \gamma\mu_1)} \sin (2\pi s)   \pmod{1}.\\ 
\end{array}
\right.
\end{equation}
Bear in mind that 
$$\dpt \lim_{\omega \rightarrow +\infty} \left[\sqrt{a_2(\omega)}\big/ \frac{e}{2\omega}\right]=1  \quad \text{and} \quad \dpt \lim_{\omega \rightarrow +\infty} \left[\sqrt{a_2(\omega)}- \frac{e}{2\omega}\right]=0,$$
 up to $\mathcal{O}(x^\delta)$--terms,  we obtain the one-dimensional reduction given in Proposition \ref{first return Cases 3 and 4} by:
\begin{equation}\label{Case4F}
\left\{
\begin{array}{l}
\mathcal{F}^1_{\gamma} (x,s)=\gamma \, \mu_1  \\ \\
\mathcal{F}^2_{\gamma} (x,s)=\dpt s+\frac{\omega \mu_3}{\pi}-\frac{\xi\omega }{\pi} \log(\gamma \, \mu_1)- \frac{ \xi \omega}{2\, e \, \mu_1 \, \pi } + \frac{ \xi}{2 \, \pi\,   \mu_1 } \sin (2\pi s)   \pmod{1}.\\ 
\end{array}
\right.
\end{equation}

\medbreak

\section{Proof of Theorem \ref{Th A}}
\label{proof Th B}
The purpose of this section is to prove the main result of this manuscript. Our starting point is the expression \eqref{Case1B.1} with  $\mu_1=1$ and $\omega \approx 0$ fixed. Recall that $\tilde\gamma = \min\{\gamma_1, \gamma_2, \gamma_3\}$.

\subsection{Change of coordinates}
\label{change_of_coordinates}
For $\gamma \in\,\,  ]0, \tilde\gamma[ $ fixed and $(x,s) \in \tilde{\mathcal{D}}$, let us make the following change of coordinates:
  \begin{equation}
  \label{change2}
x_{\,\text{new}}  \mapsto \gamma^{-\frac{1}{\delta}} x_\text{old}.
  \end{equation}
Observe that
$$
\mathcal{F}^1_{\gamma} ( x_{\,\text{new}} ,s) =  \gamma\left( \left(\frac{x_{\,\text{old}} }{\gamma^\frac{1}{\delta}}\right)^\delta +  \left[ 1- \sqrt{a_1}\cos (2\pi s) \right] \right)=   \gamma^p \left(   {x_{\,\text{new}} }^\delta +  \left[ 1- \sqrt{a_1}\cos   (2\pi {s})\right] \right),
$$
for some $p>0$. Setting   $(x_{\,\text{new}} ,s)\equiv (x, s)$, we get:
    \begin{eqnarray*}
    \mathcal{F}^1_{\gamma} ( x  ,s) &=& \gamma^p \left(   {x }^\delta +  \left[ 1- \sqrt{a_1}\cos   (2\pi {s})\right] \right) \\
\mathcal{F}^2_{\gamma} (x,s) &=& {s}+\frac{\mu_3\, \omega}{\pi} - \frac{\xi\, \omega\, p}{\pi}\ln(\gamma)- \frac{\xi \,  \omega}{\pi} \log \left(  {x}^\delta +  \left[ 1- \sqrt{a_1}\cos (2\pi {s})\right] \right). %- \frac{ \xi\omega\gamma}{2\, e\, \pi} \frac{1-A \cos(2\pi  \overline{y} -\tilde{y})}{ \gamma \left(   \overline{x}^\delta +  \left[ 1- \sqrt{a_1}\cos   2\pi\overline{y}\right] \right)}.
  \end{eqnarray*}
%and then define:
%\begin{eqnarray*}
%h_a(s):=\mathcal{F}^2_{\gamma} (0,s) &=& {s}+\frac{\mu_3\, \omega}{\pi} - \frac{\xi\, \omega\, \alpha}{\pi}\ln(\gamma)- \frac{\xi \,  \omega}{\pi} \log \left[    1- \sqrt{a_1}\cos (2\pi {s})\right].  %- \frac{ \xi\omega\gamma}{2\, e\, \pi} \frac{1-A \cos(2\pi  \overline{y} -\tilde{y})}{ \gamma \left(   1- \sqrt{a_1}\cos   2\pi\overline{y} \right)}.
%  \end{eqnarray*}

 We now view the  family of maps $\mathcal{F}_\gamma=(\mathcal{F}^1_{\gamma}, \mathcal{F}^2_{\gamma})$ as a two-parameter family of embeddings  satisfying the Hypotheses \textbf{(H1)--(H7)} of  Subsection \ref{rank_one}.

\subsection{Reduction to a singular limit}
For $\gamma \in\,\, ]\, 0, \tilde\gamma\,  [$, we compute the singular limit associated to $\mathcal{F}_\gamma$ written in the coordinates $(x,s)$ defined in Subsection \ref{change_of_coordinates}.
Let  $k: \RR^+ \rightarrow \RR$ be the invertible map defined by $$k(x)= -K_\omega \, \xi  \ln (x), \qquad \text{where} \qquad K_\omega = \frac{\omega  \, p}{\pi}>0.$$
As suggested in Figure \ref{scheme5}, define the decreasing sequence $(\gamma_n)_n$ such that, for all $n\in \NN$: \medbreak
\begin{enumerate}
\item  $\gamma_n\in\, ]0, \tilde\gamma[$ and \\
\item $k(\gamma_n) \equiv 0 \mod 1$.
\end{enumerate}

\bigbreak
Since $k$ is an invertible map,  for $a \in \EU^1 \equiv [0, 1[$ fixed and $n\geq n_0\in \NN$, let  
\begin{equation}
\label{sequence1}
\gamma_{(a, n)}= k^{-1}\left(k(\gamma_n)+a\right)\,\,  \in \,\, ]\, 0, \tilde\gamma\, [.
\end{equation}
It is immediate to check that:
\begin{equation}
\label{sequence2}
k(\gamma_{(a, n)})= -K_\omega \, \xi  \ln (\gamma_{n})+a=a \mod 1.
\end{equation}
The following proposition establishes the $C^3$--convergence  to a \emph{singular limit}, as $n \rightarrow +\infty$ (defined in the appropriate set of maps from $\tilde{\mathcal{D}}$ to $\RR$).

\begin{lemma} 
\label{important lemma}
In the $C^3$--Whitney topology, the following equality holds:
$$
\lim_{n\in \NN} \|\mathcal{F}_{\gamma_{(n,a)}}  -(\textbf{0}, h_a )\| =0$$
where  $\textbf{0}$ represents the null map and 
\begin{equation}
\label{circle map}
h_a (x,s) =  s+a - \frac{\xi\, \omega}{\pi} \log \left[  (1- \sqrt{a_1}\cos (2\pi {s}))\right] .
\end{equation}
\end{lemma}

\begin{figure}
\begin{center}
\includegraphics[height=8cm]{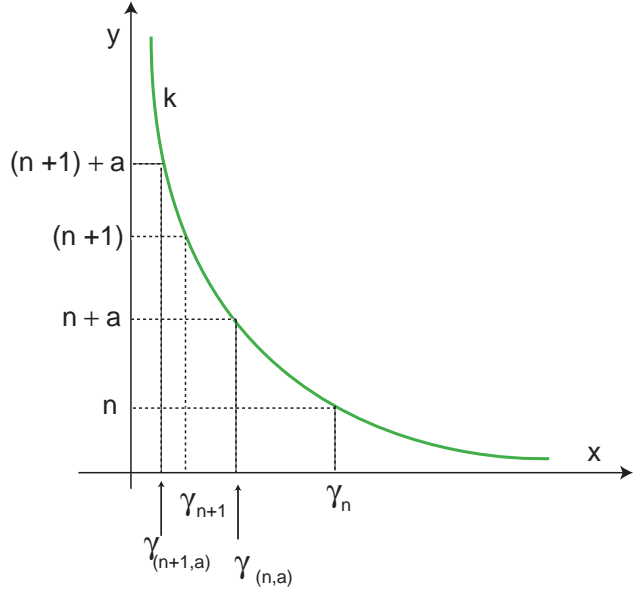}
\end{center}
\caption{\small Graph of the map $k(x)= -K_\omega \, \xi  \ln (x)$ for $\dpt K_\omega = \frac{\omega \,  p}{\pi}$, and illustration of the sequences $(\gamma_n)_n$ and $(\gamma_{(n,a)})_n$ for a fixed $a \in [0,1[$. }
\label{scheme5}
\end{figure}

\begin{proof}

Using \eqref{sequence2}, note that

 \begin{eqnarray*}
 \mathcal{F}^1_{\gamma_{(n,a)}} (x,{s}) &=&  \gamma_{(n,a)}^p \left(   {x}^\delta +  \left[ 1- \sqrt{a_1}\cos   (2\pi {s})\right] \right) \\ \\
\mathcal{F}^2_{\gamma_{(n,a)}} (x,s) &=& {s}+\frac{\mu_3\, \omega}{\pi} - K_\omega \, \xi  \ln(\gamma_{(n,a)})- \frac{\xi\, \omega}{\pi} \log \left[  {x}^\delta +  \left[ 1- \sqrt{a_1}\cos (2\pi {s})\right] \right] \\\\
&=& {s}+\frac{\mu_3\, \omega}{\pi} +a-\frac{\xi \, \omega}{\pi} \log \left[ {x}^\delta +  \left[ 1- \sqrt{a_1}\cos (2\pi {s})\right] \right].
  \end{eqnarray*}
    Therefore, since $\dpt \lim_{n\in \NN} {\gamma_{(n,a)}}=0$ we may write:

 \begin{eqnarray*}
 \lim_{n\in \NN}   \mathcal{F}^1_{\gamma_{(n,a)}} (x,{s}) &= &0 \\ \\
   \lim_{n\in \NN}   \mathcal{F}^2_{\gamma_{(n,a)}} (x,{s}) &=&  s+\frac{\mu_3\, \omega}{\pi}+a  - \frac{\xi\, \omega}{\pi} \ln \left(1- \sqrt{a_1}\cos (2\pi {s}\right)) \pmod{1}
  \end{eqnarray*}
 and we get the result.

\end{proof}

\begin{remark}
\label{rem:non-deg}
\label{rem7.2} The map $h_a(x,s)$ just depends on $s$:
$$h_a(x,s)\equiv \mathcal{F}^2_{\gamma_{(n,a)}}(0,s) = s+\frac{\mu_3\, \omega}{\pi}+a  - \frac{\xi\, \omega}{\pi} \ln \left(1- \sqrt{a_1}\cos (2\pi {s}\right)),$$  
this is why in \S \ref{ss:verification}, we identify $h_a(x,s)$ and $h_a(s)$.  The map $h_a(s)$  is a \emph{Morse function} and has finitely many nondegenerate critical points (note that $\sqrt{a_1}<1$). 
\end{remark}

\subsection{Verification of the hypotheses of the theory of rank-one maps.}
\label{ss:verification}
From now on, our focus will be the sequence of two-dimensional maps 
\begin{equation}
\label{family}
F_{(a,b)}= \mathcal{F}_{(a, \gamma(n,a))} \qquad \text{with} \qquad n \in \mathbb{N} \qquad \text{and} \qquad a \in [\, 0,1[ \text{ fixed}.
\end{equation}
 Since our starting point  is an attracting heteroclinic network (for $\gamma=0$), the \emph{absorbing sets} defined in Subsection 2.4 of \cite{WY} follow from the existence of the attracting annular region ensured by next result:
\begin{lemma}[\cite{AHL2001, TD}, adapted]
\label{lemma_annulus}
There exists $\gamma^\star>0$ small such that for all $\gamma \in \, ]\, 0, \gamma^\star\, ]$, the region 
$$
\mathcal{B}= \left\{(x, s) \in \tilde{\mathcal{D}}:\qquad 0<x^\star -2\gamma\sqrt{a_1} \leq x\leq x^\star +2\gamma\sqrt{a_1} , \qquad s \in [0, 1[ \, \right\}
$$
is forward-invariant set under $\mathcal{F}_{\gamma}$--interation, corresponding to the shaded region of Figure \ref{bif1}.
\end{lemma}
Now, we show that the family of maps \eqref{family}  satisfies Hypotheses \textbf{(H1)--(H7)} stated in Subsection \ref{rank_one}.

\medbreak
\begin{description}
\item[\text{(H1)}]  
The first two items are immediate.  We establish the distortion bound \textbf{(H1)(3)} by studying $D\mathcal{F}_{(a, \gamma(n,a))}$. Direct computation  implies that for every $\gamma \in (0, \tilde\gamma)$ and $(x,s) \in \tilde{\mathcal{D}}$, one gets:

$$D\mathcal{F}_{(a, \gamma(n,a))}({x},{s})=\left(\begin{array}{cc} \dpt   \frac{\partial \mathcal{F}_{(a, \gamma(n,a))}^1(x,s)}{\partial {x}} &\dpt   \frac{\partial \mathcal{F}_{(a, \gamma(n,a))}^1(x,s)}{\partial {s}} \\ \\ \dpt  \frac{\partial \mathcal{F}_{(a, \gamma(n,a))}^2(x,s)}{\partial {x}} \dpt & \dpt  \frac{\partial \mathcal{F}_{(a, \gamma(n,a))}^2(x,s)}{\partial {s}} \end{array}\right)
$$
where
\begin{eqnarray*}
\frac{\partial \mathcal{F}_{(a, \gamma(n,a))}^1({x}, {s})}{\partial{x}} &=& \gamma_{(n,a)}^p \, \delta {x}^{\delta-1} \\ \\
\frac{\partial \mathcal{F}_{(a, \gamma(n,a))}^1({x}, {s})}{\partial {s}}&=& 2\, \pi\,  \gamma_{(n,a)}^p \, \sqrt{a_1} \sin (2\pi {s})\\ \\
\frac{\partial \mathcal{F}_{(a, \gamma(n,a))}^2({x}, {s})}{\partial {x}} &=& -\xi\frac{\omega}{\pi} \left( \frac{\delta {x}^{\delta-1}}{ {x}^{\delta}+(1-\sqrt{a_1}\cos (2\pi {s}))} \right) \\ \\
\frac{\partial \mathcal{F}_{(a, \gamma(n,a))}^2({x}, {s})}{\partial {s}} &=& 1- \xi\frac{\omega}{\pi} \left(
 \frac{ 2\pi\, \sqrt{a_1}\sin(2\pi {s})}{{x}^\delta +\gamma (1 - \sqrt{a_1}\cos(2\pi {s}))} 
\right)\\ 
\end{eqnarray*}
and therefore
 $$|\det D \mathcal{F}_{(a, \gamma(n,a))}(x,{s})| = \gamma_{(n,a)}^p  \delta x^{\delta-1}.$$
Since  $x>0$ (Lemma \ref{lemma_annulus}),  we conclude that there exists $\gamma^\star>0$ small enough such that:
$$
\forall \gamma \in \, ]\, 0, \gamma^\star \, [, \qquad \det D \mathcal{F}_\gamma(x,{s})|  \in  \,\,  ]\, k_1^{-1}, k_1\, [ ,
$$
for some $k_1>1$. This implies that Hypothesis \textbf{(H1)(3)}  is satisfied.
\bigbreak
\item[\text{(H2) and (H3)}] It follows from Lemma \ref{important lemma} where $b=\gamma_{(n,a)}$. 

\bigbreak
\item[\text{(H4) and (H5)}] 
These hypotheses are connected with the family of circle maps $$h_a: \EU^1 \rightarrow \EU^1$$ defined in Remark \ref{rem7.2}, noting that $ \mathcal{F}^2_{\gamma_{(n,a)}}(0,s)\equiv h_a(s)$.
Taking into account Proposition 2.1 of \cite{WY2003}, there exists $\xi^\star \in \RR^+$ such that if $\xi>\xi^\star$, the family $$
h_a(s)=s+\frac{\mu_3\, \omega}{\pi} +a-  \frac{\xi\, \omega}{\pi} \log(1-\sqrt{a_1} \cos(2\pi s)) \qquad a\in [\, 0,1\, [
$$
satisfies Properties \textbf{(H4)} and \textbf{(H5)}.

\bigbreak

\item[\text{(H6)}] The computation follows  from direct computation using  the expression of $\mathcal{F}_\gamma(x,s)$. Indeed, for each $s\in C_{a^\star}$ (set of critical points of $h_{a^\star}$), we have
$$
\frac{d}{dx} \mathcal{F}_{(a^\star,0)}(x,s) |_{x=0} =1.
$$

\bigbreak

\item[\text{(H7)}] It follows from Proposition 2.1 of \cite{WY2003} if $\xi $ is large enough ($\Rightarrow$ the ``big lobe'' of \cite{TS1986} is large enough).

\end{description}

\bigbreak

Since the family $\mathcal{F}_{(a, \gamma(n,a))}$ satisfies \textbf{(H1)--(H7)} then, for $\gamma^+=\min\{\tilde\gamma,  \gamma^\star\}>0$, if $\xi>\xi^\star$,  there exists a subset $\Delta \in  [0, \gamma^+]$ with positive Lebesgue
measure such that for $\gamma\in \Delta$, the map $\mathcal{F}_{\gamma}$ admits a strange attractor $$ \Omega \subset \bigcap_{m=0}^{+\infty}  \mathcal{F}_{\gamma}^m(\tilde{\mathcal{D}})$$   supporting a unique ergodic SRB measure $\nu$. The orbit of Lebesgue almost all points in ${\Omega}$ has positive Lyapunov exponent  and is asymptotically distributed according to $\nu$.  The \emph{abundance} of strange attractors follows from \S3 of \cite{WO}.

\begin{remark}
 The strange attractor $\Omega$ is non-uniformly hyperbolic, non-structurally stable and is the limit of an increasing sequence of uniformly hyperbolic invariant sets.
\end{remark}

The proof of \cite{WY2001} goes further. The pair $(\mathcal{F}_{\gamma}, \nu) $ has \emph{exponential decay of correlations} for Holder continuous observables: given an Holder exponent $\eta$, there exists $\tau = \tau(\eta)<1$ such that for all Holder maps $\varphi$, $\psi: \Omega \rightarrow \RR$ with Holder exponent $\eta$, there exists $K(\varphi, \psi)\geq 0$ such that for all $m\in\NN$, we have:
$$
\left|\int (\varphi \circ \mathcal{F}_{\gamma}^{m})\psi \,  d\nu - \int \varphi d\nu \int \psi d \nu\right| \leq K(\varphi,\psi) \tau^m.
$$
%The analysis of the ergodic consequences of this inequality is deferred to future work because it is beyond the scope of the present work.

\iffalse

$$
h_\omega(s)= s + \frac{\omega}{\pi}{\nu} + \frac{\xi}{4\pi \mu_1}\sin(2 \pi s)
$$

$\delta(f_\omega)= \sup_{\theta_1<\theta_2} (h_\omega(\theta_1)-h_\omega(\theta_2) + \theta_1 - \theta_2$
where $\theta_1, \theta_2$ are the critical points of $h_\omega$

\begin{lemma}
$\lim_{\omega \rightarrow +\infty}\delta(h_\omega)=+\infty$ and $\lim_{\omega \rightarrow 0^+}\delta(h_\omega)=0$
\end{lemma}

It is easy to see that 
$$
h_\omega(s)= 1- \frac{2\omega \xi \sqrt{a_1}\sin(2\omega s)}{1-\sqrt{a_1}\cos(2\omega s}
$$
In particular
$h'(s)=0$ if and only if
$$
\sqrt{a_1}\cos(2\pi s) + 2 \omega \xi \sqrt{a_1}\sin(2\pi s)=1.
$$
The two critical points are $s_1= \pi/2$ and $s_2= 3\pi/2$. Therefore
$$
\delta(h_\omega)= h_\omega (1/2) - h(3/2) - 1 = -\frac{\omega}{\pi} \log\left(\frac{1-\sqrt{a_1}}{1+\sqrt{a_1}}\right)
$$
Since $a_1= \frac{c^2}{c^2+ 4\omega^2}$, it follows that $\lim_{\omega \rightarrow +\infty}\delta(h_\omega)=+\infty$.
\fi

\section{From the attracting torus to strange attractors: \\ a geometrical interpretation}
\label{s:mechanism}
 In this section, we give a geometrical interpretation of the mechanisms behind the creation of \emph{rank-one strange attractors}. We also compare our results with previous works in the literature.  
 \medbreak
With respect to the original equation \eqref{general},  borrowing the ideas of \cite{AHL2001}, for  $\omega>0$ fixed,  we may draw  two smooth curves, the graphs of $t_1$ and $t_2$ shown in Figure \ref{graphs1}, such that:
\begin{enumerate}
 \medbreak
\item  $\dpt t_1(\xi) = \frac{\exp(C/\xi\omega)-1}{ \exp(C/\xi\omega)-1/C}$, $C>2$, and $\dpt t_2(\xi)= 1/\sqrt{1+(\xi\omega)^2}$;
\medbreak
\item the region above the graph of $t_1$ corresponds to vector fields whose flows exhibit  suspended \emph{rotational horseshoes} \cite{PPS} (compare with Theorem \ref{Th3.2} in Section \ref{overview});
\medbreak
\item the region below the graph of $t_2$  corresponds to flows having an invariant and attracting torus with zero topological entropy.
\end{enumerate}

\begin{figure}[ht]
\begin{center}
\includegraphics[height=6cm]{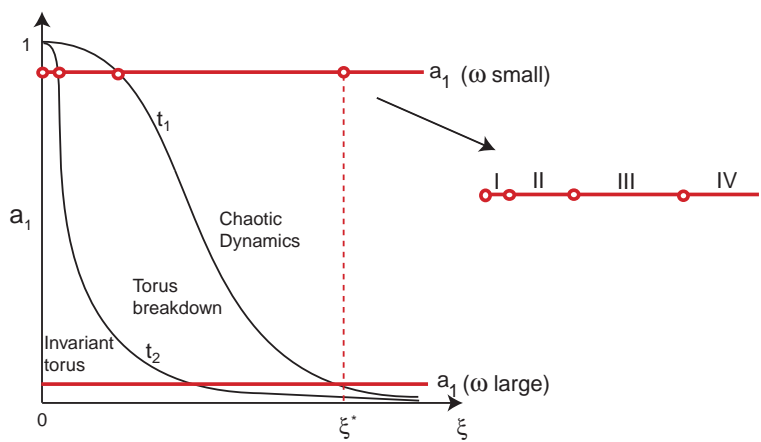}
\end{center}
\caption{\small  Graphs of $t_1$ and $t_2$ and the relative position of $a_1$ for $\omega$ small  and   large (\emph{cf.} Lemma \ref{lemma1}). I - attracting curve; II - torus-breakdown bifurcations associated to an Arnold tongue; III - Rotational horseshoes; IV - rank-one strange attractors.  }
\label{graphs1}
\end{figure}

\bigbreak

From now on, we focus on Cases 3 and 4 of Table 2 ($\omega \approx 0$).   For $\delta \gtrsim1$, if  $\xi$ is small enough then the initial deformation of the singular limit  is suppressed by the ``contracting force''  and the maximal attracting set is a non-contractible closed curve satisfying Afraimovich's Annulus Principle \cite{AHL2001}. This is consistent with the results stated in \cite{TD2} about the existence of an attracting curve for the map $\mathcal{F}_\gamma$.  The parameters considered in \cite{DT3} do not allow to see \emph{observable chaos} since the parameter $\dpt \xi=\frac{1+\delta+\delta^2}{e}$ for $\delta\gtrsim 1$,  cannot be large enough.% 

If $\delta\gg1$ (region I of \cite{TD2}), then  $\xi$ may be either small or large.
If  $\xi$ is small, the flow of  \eqref{general} has again an attracting normally hyperbolic torus. If   $\xi$ is large, then the initial deformation brought by the perturbing term is exaggerated. The attracting region $\mathcal{B}$ (Lemma \ref{lemma_annulus}) starts to disintegrate into a finite collection of periodic saddles and sinks, a phenomenon occurring within an Arnold tongue   \cite{Aronson}. These bifurcations correspond to what the authors of \cite{WY} call \emph{transient chaos} associated to the \emph{Torus-breakdown bifurcations} \cite{AS91, Aronson, Rodrigues2019}.   
The curves $\mathcal{H}_n$ of Figure 14 of \cite{TD} indicate the location of homoclinic bifurcations involving the $n$-period orbit within the corresponding  Arnold tongue. As discussed in Section 5 of \cite{Boyland}, a \emph{homoclinic bifurcation} occurs when the set of preimages of an unstable period-$n$ orbit of $  h_a(s) $ contains a critical point (minimum or maximum), giving rise to chaotic dynamics, not necessarily observable.

  As  $\xi$ gets larger ($\xi>\xi^\star$), the initial deformation introduced by the perturbing term is exaggerated further, getting us the emergence of rank-one attractors (Theorem \ref{Th A}), obtained by \emph{stretch and fold}. These two mechanisms are due to the   attracting features of $\Gamma$ combined with the presence of non-degenerated turns of the circle map $h_a$. Points of the singular limit ensured  by \textbf{(H2)}, at different distances from $x=0$, rotate at different speeds. As $\gamma$ varies, the two critical values of $h_a$ move at rates $1/\gamma$ in opposite directions. 
   See an illustration of this mechanism in Figure \ref{bif1}.

\begin{figure}[ht]
\begin{center}
\includegraphics[height=8.0cm]{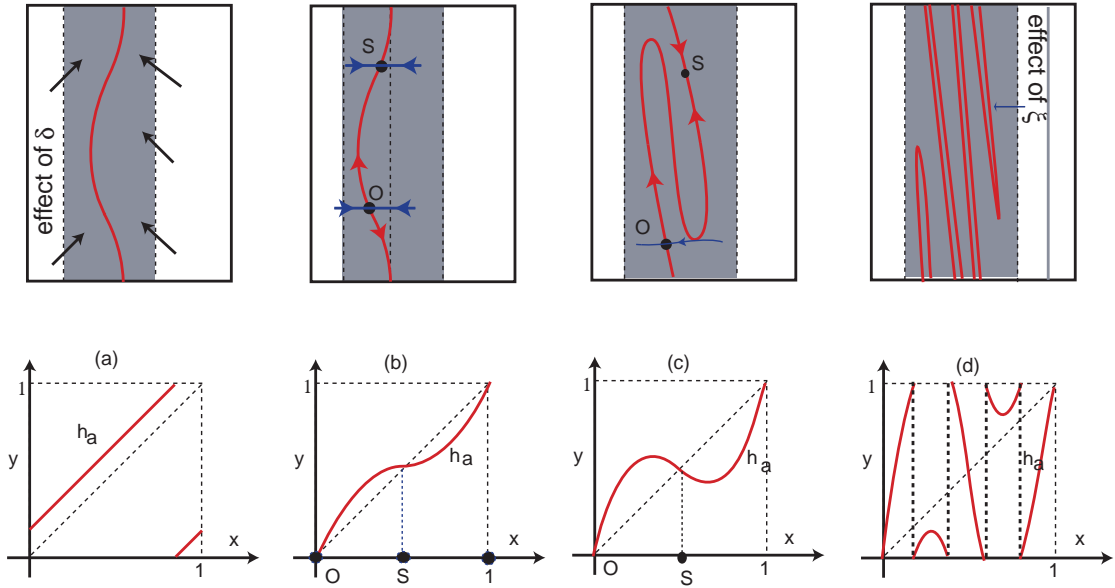}
\end{center}
\caption{\small  Illustration of the emergence of rank-one attractors for $\mathcal{F}_\gamma$ from the dynamics of $h_a(s) \pmod{1}$ when $\xi$ varies. (a) $\xi= \xi_0\gtrsim 0$ -- attracting torus. (b) $\xi=\xi_1$ -- attracting torus inside resonance tongue with a saddle and a sink. (c)   $\xi=\xi_2$ -- torus-breakdown.  (d) $\xi=\xi^\star$  -- mixing properties (isomorphic to a Bernoulli shift).  }
\label{bif1}
\end{figure}

For $\omega \approx 0$ and $\delta \gg1$, the dynamics of the first return map is chaotic and no longer reducible to a one dimensional map. Nevertheless, according to \cite{WY}, certain ``good'' properties of the singular cycle may be passed back to the two-dimensional system (see Remark \ref{rem:non-deg}).  Forgetting temporarily its connection to  equation \eqref{circle map}, we might  think of $h_a$ as an abstract circle
map. 

\begin{itemize}
\item If $h_a$ is a diffeomorphism, the classical theory by Denjoy
may be applied. We point out a  resemblance between ``our''
$h_a$ and the  family of circle maps first studied in \cite{Aronson}. Because of strong normal contraction, invariant curves are shown to exist independent of
rotation number.  
\item If $h_a$ is not invertible, two types of dynamical behaviours are known to be \emph{prevalent}. There is some evidence that these are only two observable pure
dynamics types:  maps with \emph{sinks} or maps with \emph{absolutely continuous
invariant measures}. Wang and Young's theory  \cite{WY} provides the bridge from the non-invertible circle maps theory to dissipative systems of the form \eqref{Case1B.1}. 
\end{itemize}

 For $\gamma>0$ fixed, the evolution of $h_a$ as $\xi$ varies, is suggested in Figure \ref{bif1}. In (a) and (b), we see the existence of an invariant curve (giving rise to an attracting torus). 
In case (b) we may see the existence of two fixed points, suggesting that the chosen parameters  are within a resonant wedge \cite{Aronson}. In (c), the map $h_a$ is not a diffeomorphism meaning that the invariant torus is broken; it corresponds to the point when the unstable manifold of the saddle $O$ (in the Arnold tongue \cite{Aronson}) turns around.  In case (d), the Property \textbf{(H7)} holds, meaning that the unstable manifold of the saddle $O$  crosses each leaf of the stable foliation of other saddles of the torus's ghost.  Figure 4 of \cite{WY} is particularly suggestive to understand this phenomenon. 
 As a conclusion, we may give a geometrical interpretation of the parameters in our context (see Figure \ref{bif1}):
 
 \begin{eqnarray*}
\delta>1 &\mapsto & \text{forces the existence of an invariant attracting region $(\tilde{\mathcal{D}}$);} \\ 
\xi\gg 0 &\mapsto & \text{forces the large number of turns of the singular limit}  \\
& & \text{plays the same role as the \emph{twisting number} of \cite{Rodrigues2019};}\\   
\omega>0 &\mapsto & \text{governs the amplitude of the non-autonomous perturbation of } \mathcal{F}_\gamma  \\ 
&& \text{(recall  that $a_1$ depends on $\omega$);} \\ 
\gamma>0 &\mapsto & \text{inverse of the dissipation.} \\ 
\end{eqnarray*}

\section{Discussion}
\label{s:Discussion1}

In this article, we have discussed the dynamics of a periodically-perturbed vector field in $\RR^3$ whose unperturbed flow has a symmetric  and clean attracting heteroclinic network. This work should be seen as the natural continuation of \cite{TD}. 

Based on the numerics of \cite{DT3, TD2}, we distinguish four cases for the dynamics. In the case $\delta \gg 1$ and $\omega \approx 0$, we have refined the analysis of \cite{TD} and introduced a new parameter $\xi$.
  Taking into account the action of this new parameter, we have formulated a checkable hypothesis under which the map $\mathcal{F}_\gamma$, induced by the flow of the forced system, admits a strange attractor. It supports a unique ergodic SRB measure for a set $\Delta$ of forcing amplitudes  with $Leb(\Delta)>0$.  For all $\gamma\in \Delta$, the flow-induced map is \emph{rank-one} in the sense of \cite{WY};  the chaos is observable and abundant. This  result gives  a rigorous answer to the problem raised in Section 8 of \cite{MO15}. 
  Before finishing the paper, we would like to stress the following two remarks: 
  \begin{enumerate}
  \item  Conditions \textbf{(C1a)} and \textbf{(C1b)} are the only hypotheses we really need to prove Theorem~\ref{Th A}. Condition \textbf{(C2)} simplifies the computations but  it may be relaxed. All results are valid if $1- \sqrt{a_1}\cos (2\omega \Psi(0,s))$ (cf. \eqref{eq1}) is a  positive Morse function with finitely many non-degenerate points, which is a generic requirement. See Proposition 2.1 of~\cite{WY2003}.

  \item Although the non-autonomous perturbation term only acts on the $x$-coordinate  in (\ref{general}), the whole calculation can be carried out in the same way for any non-negative periodic forcing acting on all components.  The use of other periodic forcing would lead qualitatively similar dynamical regimes.  The only necessary requirement is that the external periodic forcing should be non-negative with small amplitude $\gamma$.
   \end{enumerate}

In summary, the forced May-Leonard system \eqref{general} or, equivalently, the forced Guckenheimer and Holmes system, may behave  periodically, quasi-periodically or chaotically, depending on specific character of the  forcing. Ergodic consequences of this article are in preparation.

\section*{Acknowledgments}
 The author would like to express his gratitude to Isabel Labouriau for helpful discussions. The author is also grateful to the two referees for the constructive comments, corrections and suggestions which helped to improve the readability of this manuscript.

\end{document}